\documentclass[11pt]{article}
\usepackage{amsmath, amscd, amssymb, latexsym, epsfig, color, amsthm, tikz}
\usetikzlibrary{positioning,shapes.geometric}
\usepackage{verbatim}
\usepackage{float}
\usepackage{booktabs}
\usepackage[hidelinks]{hyperref}
\numberwithin{equation}{section}
\newtheorem{theorem}{Theorem}[section]
\newtheorem{proposition}[theorem]{Proposition}
\newtheorem{corollary}[theorem]{Corollary}
\newtheorem{lemma}[theorem]{Lemma}
\theoremstyle{definition}
\newtheorem{definition}[theorem]{Definition}

\newtheorem{question}[theorem]{Question}
\newtheorem{conjecture}[theorem]{Conjecture}
\newtheorem{remark}[theorem]{Remark}
\DeclareMathOperator{\conv}{\mathrm{conv}}
\DeclareMathOperator{\aff}{\mathrm{aff}}
\DeclareMathOperator{\supp}{\mathrm{supp}}
\newcommand{\R}{{\mathbb R}}

\title{\LARGE
On Unavoidable Faces of High-Dimensional Polytopes}
\author{
	Jes\'us A. De Loera
	\qquad
	Ethan X. Fang
	\qquad
     Shengtao Guo
	\\[1ex]
	Junwei Lu
	\qquad
	Hailun Zheng
}
\date{}

\begin{document}
	\maketitle
\begin{abstract}
Kalai's cube--simplex conjecture asserts that for all positive integers $\ell,k$, there is an integer $f(\ell,k)$ such that every polytope of dimension at least $f(\ell,k)$ has either a simplex $\ell$-face or a cube $k$-face; let $f_s(\ell,k)$ denote the threshold restricted to simple polytopes. Finiteness of $f(\ell,k)$ is known only for $\ell,k \leq 2$. 
In addition, Kalai proved that $f_s(2,k) \leq 2k^2$.
Here we prove that $f_s(\ell,k)$ is finite for all $\ell \geq 2$ and $k \geq 3$, the first such result beyond $\ell = 2$, with $f_s(2,k) \leq 2k^2-1$ and $f_s(\ell,k) \leq \tfrac{1}{2}k^2\ell\,2^k$ for $\ell \geq 3$. In the opposite direction, we obtain the lower bounds
$f(\ell,k) \geq (5\lfloor \ell/2 \rfloor + (\ell \bmod 2) - 1)(k-1)+1$ and
$f_s(\ell,k) \geq \max\{4,\,2(\ell-1)\}(k-1)+1$.

A companion question asks for the minimum possible size of a 3-face within a higher-dimensional polytope. Meisinger, Kleinschmidt and Kalai proved that every rational $d$-polytope with $d \geq 9$ has a $3$-face with fewer than $78$ vertices or fewer than $78$ facets. Here we improve their bound: every convex polytope of dimension at least $15$ has a $3$-face with at most $13$ facets. One step of our proof requires an explicit exact rational certificate or identity on flag numbers. This certificate is computed using linear programming. 
\end{abstract}
\section{Introduction} 

%A classical consequence of Euler's formula for $3$-polytopes (dating back in dual form to Descartes \cite{Descartes1897,Federico1982}) is 
Faces of polyhedra has been an old pursuit of study in Geometry (see \cite{Grunbaum,HRZ} for an introduction to this old rich area). Nevertheless, simple natural questions remain unanswered: One can easily see that every $3$-polytope contains a $2$-face with $3$, $4$, or $5$ vertices \cite{Grunbaum}, 
but it is clear that the regular dodecahedron in three dimensions and the regular $120$-cell ($d=4$) have only pentagonal $2$-faces. This led Perles and Shephard \cite{PerlesShephard1967} to ask whether $d$-polytopes for $d \ge 5$ can avoid triangles and quadrilaterals entirely. More generally, one can wonder what faces must appear as the dimension of the polytope grows? 

Kalai \cite{Kalai} answered Perles and Shephard's question by proving that every $5$-polytope contains a $2$-face with $3$ or $4$ vertices and initiated a research path investigating the low-dimensional substructures that must appear in high-dimensional polytopes. It is remarkable that the proof uses known linear equations and inequalities about face vectors and flag vectors, which lead to an unsolvable system without integer solutions. This result had also fascinating consequences: there is no face-to-face tiling of $\mathbb{R}^d$ ($d \ge 5$) by regular cross-polytopes, sharpening results of Schulte \cite{Schulte1984}. Similarly, Gromov's ``no $\Delta$ no $\square$'' negative curvature condition for cubical complexes cannot hold for locally finite cubical manifolds of dimension $d \ge 5$ \cite{Gromov1987}.
Motivated by his work, Kalai later formulated the following precise definitions, a key global question, and a concrete conjecture in the same flavor as Ramsey's theorems:

\begin{definition}[\cite{Kalai}]
For integers $\ell, k \ge 1$:
\begin{itemize}
    \item $f(\ell, k)$ is the smallest integer such that every $d$-polytope with $d \ge f(\ell, k)$ contains either a simplex face $\sigma^\ell$ or a cube $k$-face $C^k$. If no such integer exists, $f(\ell, k) = \infty$.
    \item $f_s(\ell, k)$ is the corresponding threshold integer restricted to the class of \emph{simple} $d$-polytopes.
\end{itemize}
\end{definition}
		
\begin{question}
		Fix integers $\ell, k\geq 2$. Are $f_s(\ell,k)$ and $f(\ell,k)$ bounded? 
\end{question}

\begin{conjecture}[Kalai's Cube-Simplex Conjecture \cite{Kalai}]
For all integers $\ell, k \ge 1$, $f(\ell, k) < \infty$.
\end{conjecture}

Some results are known for special families of polytopes. For instance, McMullen \cite{McMullen1971} proved that every $d$-zonotope contains a $\lfloor (d+1)/2 \rfloor$-face that is a parallelotope (combinatorially a hypercube). Similarly, a simple application of Ramsey's theorem guarantees the presence of simplex or hypercube faces \cite{Chvatal1975} inside vertex-packing polytopes (indeed, all this research has the same flavor as Ramsey's theorem guaranteeing certain graphs must appear inside sufficiently large graphs).
But up until now one can summarize all existing bounds in a short table, see Table \ref{tab:bounds}.

\begin{table}[htbp]
\centering
\begin{tabular}{@{}llll@{}}
\toprule
\textbf{Case / Class} & \textbf{Bound / Value} & \textbf{Face Obtained} & \textbf{Reference} \\
\midrule
$f(1, 1)$ & $1$ & Edge & Trivial \\
$f(2, 2), f_s(2, 2)$ & $5$ & Triangle or Quadrilateral & Kalai (1990) \cite{Kalai} \\
$f_s(2, k)$ (Simple) & $\le 2k^2$ & Triangle or $k$-Cube & Kalai (1990) \cite{Kalai} \\
$f(\ell, k)$ for $\ell,k \geq 3$ & Unknown & $\ell$-simplex or $k$-cube & \textbf{Open} \cite{Kalai} \\
\bottomrule
\end{tabular}
\caption{Known status for face-containment thresholds in high dimensions.}
\label{tab:bounds}
\end{table}

Specifically about simple polytopes, Kalai had earlier used Nikulin's 
average face formulas \cite{Nikulin} and results of Blind and Blind \cite{BlindBlind1984,BlindBlind}, to establish bounds for simple polytopes:
\begin{theorem}[Kalai \cite{Kalai}]
Let $P$ be a simple $d$-polytope:
\begin{enumerate}
    \item If $d \ge 2k - 1$, $P$ has a $k$-face with fewer than $(k+1)(k+2)$ facets.
    \item If $d \ge 2k^2$, $P$ has a $k$-face with at most $2k$ facets.
    \item If $P$ has no triangular $2$-faces and $d \ge 2k^2$, then $P$ contains a $k$-face combinatorially isomorphic to a $k$-cube. In particular, $f_s(2, k) \le 2k^2 < \infty$.
    \item For any $\varepsilon > 0$, if $d > 2k^2/\varepsilon$, then more than $(1 - \varepsilon)f_k(P)$ of all $k$-faces of $P$ are isomorphic to $k$-cubes.
\end{enumerate}
\end{theorem}

\paragraph{Our Contributions:} Our paper improves the state of the art with three theorems:
\begin{itemize}

\item First, by generalizing Kalai's original argument that $f_s(2,k)$ is bounded \cite{Kalai}, we can show (see Section \ref{sec:upperboundfs}, in particular Corollary \ref{cor2: f_s(l,k)} for details)

\renewcommand{\thetheorem}{A}
\begin{theorem}
$f_s(\ell,k)$ is bounded for all $\ell\geq 2$ and $k\geq 3$. More precisely, we can prove that
$$f_s(\ell, k)\le \begin{cases}
2k^2-1 & \ell=2\\
\frac{1}{2}k^{2}\ell\,2^{k} & \ell\geq 3
\end{cases}.$$

\end{theorem}
\addtocounter{theorem}{-1} 

\item Second, in Section \ref{sec:lowerbound} we establish new lower bounds on $f_s(\ell, k)$. See Theorem \ref{thm:general} and Theorem \ref{thm:simple} for details. This can be compared with a result of Pfeifle showing on the basis of the Wythoff construction that $f(k, k)> (2k-1)(k-1)+1$ for $k\geq 3$ \cite[page 518]{Kal17}.

\renewcommand{\thetheorem}{B}
\begin{theorem}For all $\ell,k\ge2$,
    \[f(\ell, k) \ \ge\ (5\lfloor \ell/2\rfloor +(\ell \bmod 2)-1)(k-1)+1,\]
\[\text{and}\quad
f_s(\ell,k)\ \ge\ \max\{4,\ 2(\ell-1)\}\,(k-1)+1 .
\]
\end{theorem}
\addtocounter{theorem}{-1} 
\item Third, so far we have discussed  simplex or cube faces, but a natural question is to try to find the smallest size k-faces forced to appear as the dimension grows. For example, Meisinger, Kleinschmidt, and Kalai \cite{Meisinger2000,MKK-2000} proved, via the software \textsc{Flagtool} and toric $g_3$-invariants \cite{Stanley1987}, that every rational $d$-polytope with $d \ge 9$ contains a $3$-face with fewer than $78$ vertices or fewer than $78$ facets. Their proof combines linear programming with known linear inequalities about face vectors and flag vectors. In  Section \ref{sec:linearprogramming} we use similar tools to prove 

\renewcommand{\thetheorem}{C}
\begin{theorem}
 Every convex polytope of dimension $d \geq 15$ has a 3-face with at most 13 facets.
\end{theorem}
\addtocounter{theorem}{-1}

\end{itemize}

Note that unlike \cite{Meisinger2000,MKK-2000} we can state a theorem valid for all polytopes (not necessarily rational) thanks to the work of Karu \cite{Karu2004}. But at the same time one should contrast our results with pathological constructions that are possible when one considers general spheres versus just polytopes, see e.g., \cite{ADIPRASITO-2017}.

\noindent {\bf Notation:} In what follows we will let $s(P)$ denote the largest dimension of a simplex face of $P$, and let $c(P)$ be the largest dimension of a cube face of $P$; here $P$ itself counts as a face. Throughout, denote by $\Sigma$ the $120$-cell, by $\sigma^d$ the $d$-simplex and by $C^d$ the $d$-cube, with the convention $\sigma^{-1}=C^{-1}=\varnothing$.

\section{Boundedness of $f_s(\ell, k)$: Proof of Theorem A} \label{sec:upperboundfs}

As we said earlier Kalai used the following idea: A theorem of Blind--Blind \cite{BlindBlind} says that a $k$-face that is not a cube and has no triangular $2$-face has at least $2k+1$ facets; by a theorem of Nikulin \cite{Nikulin}, the average number of facets of a $k$-face of a simple $d$-polytope is bounded by a quantity that decreases as $d$ grows. In high dimensions, the two are incompatible, which settles the case $\ell=2$.

What we do here is to extend Kalai's idea from $\ell=2$ to general $\ell$ using  an additional ingredient. The hypothesis that $P$ has no simplex $\ell$-face does not forbid triangular $2$-faces outright. We therefore establish an upper bound on the number of triangular $2$-faces of such a polytope (Proposition \ref{prop: UB for triangles}), and run again Kalai's comparison with this bound entering as an error term.

%The proof compares two estimates for the average number of facets of a $k$-face of a simple $d$-polytope $P$. A $k$-face can have {\em few} facets in only two ways: it can have close to $k+1$ facets, which forces it to contain triangular $2$-faces (Lemma~\ref{lm: simple polytope property}), or it can be a cube, which has $2k$ facets and no triangular $2$-face at all; by a theorem of Blind--Blind \cite{BlindBlind}, every other $k$-face has at least $2k+1$ facets. Our two hypotheses rule out both: having no simplex $\ell$-face bounds the number of triangular $2$-faces above (Proposition~\ref{prop: UB for triangles}), and excluding cube $k$-faces rules out the other exception, so on average a $k$-face must have {\em many} facets. But a theorem of Nikulin \cite{Nikulin} bounds this average above by a quantity that shrinks as $d$ grows, which is impossible once $d$ is large.

For a $d$-polytope $P$ and $0\leq k\leq d-1$, let $\mathcal{F}_k(P)$ be the set of $k$-faces of $P$. Denote by $\tau_i(P)$ the number of simplex $i$-faces. In particular, $\tau_i(P)=f_i(P)$ for $i=0,1$. Throughout, we will consider the flag numbers 
\begin{equation*}
\begin{aligned}
f_{ik}(P)
&=\#\{(G,F):G\in\mathcal{F}_i(P),\ F\in\mathcal{F}_k(P),\ G\subseteq F\},\\
\tau_{ik}(P)
&=\#\{(G,F):G\cong\sigma^i,\ F\in\mathcal{F}_k(P),\ G\subseteq F\}.
\end{aligned}
\end{equation*} 
If $P$ is furthermore simple, then every $i$-face belongs to exactly $\binom{d-i}{k-i}$ $k$-faces. Hence double counting gives
\begin{equation}\label{eq: face_flag}
	f_{ik}(P)=\binom{d-i}{k-i}f_i(P), \quad \tau_{ik}(P)=\binom{d-i}{k-i}\tau_i(P).
\end{equation}

We begin by establishing several properties of simple polytopes. This first lemma gives a sufficient condition for a simple polytope to be a simplex.

\begin{lemma}\label{lm: adjacent simplex facet}
	Let $d\ge3$ and let $P$ be a simple $d$-polytope with two adjacent simplex facets. Then $P$ is a $d$-simplex.
\end{lemma}
\begin{proof}
	Assume that the adjacent simplex facets are $A=\conv\{1,\dots,d\}$ and $B=\conv\{1,\dots,d-1,d+1\}$. In particular, vertices $1,\dots,d+1$ are affinely independent. Let $T:=\conv\{1,\dots,d+1\}\subseteq P$ and $H_j:=\aff\big([d+1]\setminus\{j\}\big)$, a hyperplane for each $j\in[d+1]$.
	
	Fix $1\le i\le d-1$. As $i\in A\cap B$ and $P$ is simple, the $d$ neighbours of $i$ are $[d+1]\setminus\{i\}$. For $j\neq i$, the facet containing $i$ but avoiding the edge $ij$ contains $[d+1]\setminus\{j\}$, hence equals $P\cap H_j$. Let $H_j^+$ be the closed halfspace bounded by $H_j$ containing $P$. Then
	\[ T\subseteq P\subseteq\bigcap_{1\le j\le d+1}H_j^+=T. \]
	Thus $P=T$, i.e., $P$ is a $d$-simplex.
\end{proof}

Our second lemma gives a lower bound on the number of facets in a simple polytope in terms of the triangular 2-faces. We will need the following well-known theorem of Blind-Blind \cite{BlindBlind}.
\begin{theorem}[Blind-Blind 1990] \label{thm: Blind-Blind}
	Let $P$ be a $d$-polytope with no triangular 2-faces. Then $f_i(P)\geq f_i(C^d)$ for all $0\leq i\leq d-1$. Furthermore, $P$ is a $d$-cube if equality holds for some $0\leq i\leq d-1$.
\end{theorem}

\begin{lemma}\label{lm: simple polytope property}
	Let $k\ge3$ and let $P$ be a simple $k$-polytope which is not a cube. Then $f_{k-1}(P)\geq 2k+1-\tau_2(P)$.
\end{lemma}
\begin{proof}
We claim that, in proving the statement, we can assume $f_{k-1}(P) \leq 2k-1$.
If $f_{k-1}(P)\ge 2k+1$, then, since $\tau_2(P)\ge 0$ always, $f_{k-1}(P)\ge 2k+1\ge 2k+1-\tau_2(P)$ trivially. On the other hand, if $f_{k-1}(P)=2k$: it suffices to show $\tau_2(P)\ge 1$, since then $f_{k-1}(P)=2k\ge 2k+1-\tau_2(P)$. Suppose instead $\tau_2(P)=0$, i.e.,\ $P$ has no triangular $2$-face. Apply Theorem~\ref{thm: Blind-Blind} with $d=k$ and $i=k-1$: it gives $f_{k-1}(P)\geq f_{k-1}(C^k)=2k$, and its equality case states that if equality holds for \emph{some} index $i$, then $P$ must be combinatorially a $k$-cube. Here $f_{k-1}(P)=2k=f_{k-1}(C^k)$, so equality does hold at $i=k-1$, forcing $P\cong C^k$. This contradicts the hypothesis that $P$ is not a cube. Hence $\tau_2(P)\ge 1$ after all.

Next, let $v$ be a vertex of $P$. Since $P$ is a simple $k$-polytope, assume that $u_1, \dots, u_k$ are the neighbors of $v$, and let $H_i$ be the facet where $v\in H_i$, $u_j\in H_i$ for $j\neq i$, and $u_i\notin H_i$. We claim that for each $i$, there is a unique facet $G_i$ with $u_i\in G_i$ and $v\notin G_i$. 
Indeed, since $P$ is simple, the vertex $u_i$ lies on exactly $k$ facets. For each $j\neq i$, the facet $H_j$ contains all neighbors of $v$ except $u_j$; since $i\neq j$, this means $u_i\in H_j$. Thus all $k-1$ facets $H_j$, $j\neq i$, pass through $u_i$, and they are pairwise distinct (being $k-1$ of the $k$ distinct facets at $v$). Since $u_i$ lies on exactly $k$ facets in total, there is exactly one facet at $u_i$ other than $\{H_j : j\neq i\}$; call it $G_i$. 
This gives existence and uniqueness of a facet through $u_i$ distinct from the $H_j$'s, so it only remains to verify $v\notin G_i$. Suppose for contradiction $v\in G_i$. The facets containing $v$ are, by definition, exactly $H_1,\dots,H_k$, so $G_i$ must equal some $H_\ell$. Since $G_i\neq H_j$ for $j\neq i$ by construction, we must have $\ell=i$, i.e.\ $G_i=H_i$. But $H_i$ was defined as the facet at $v$ \emph{excluding} $u_i$, while $u_i\in G_i$ by construction — a contradiction. Hence $v\notin G_i$.
	
	If $G_i=G_j$ for some $i\ne j$, then $u_i$ and $u_j$ both lie in the $k-1$ facets $\{H_\ell:\ell\neq i, j\}\cup\{G_i\}$, whose intersection is an edge since $P$ is simple. Hence $u_iu_j$ is an edge and thus $u_iu_jv=\bigcap_{\ell\neq i, j}H_\ell$ is a triangle. Assume that $G_1, \dots, G_r$ are the distinct facets among the $G_1,\dots,G_k$. In particular, $r\leq f_{k-1}(P)-k$ as $G_i$'s are distinct from the $k$ facets $H_1,\dots,H_k$ containing $v$. 
    
    Therefore, the number of triangular $2$-faces containing $v$ is at least
	\begin{equation*}
		\begin{split}
			\#\{\{i, j\}: G_i=G_j, \;1\leq i<j\leq k\}&=\sum_{i=1}^r\binom{\#\{j: G_j=G_i\}}{2}\geq \sum_{i=1}^r \left(\#\{j: G_j= G_i\}-1\right)\\
			&=k-r\geq k-(f_{k-1}(P)-k)=2k-f_{k-1}(P).
		\end{split}
	\end{equation*}
	Summing over all the vertices, we obtain that
	\begin{equation*}\label{eq: local triangle}
		\begin{split}
			3\tau_2(P)&=\sum_{v} \#\{\text{triangular 2-faces containing $v$}\}\\
			&\geq \sum_{v} (2k-f_{k-1}(P))\ =\ f_0(P)(2k-f_{k-1}(P))\\   
			&\stackrel{(*)}{\geq} \big((k-1)f_{k-1}(P)-(k+1)(k-2)\big) (2k-f_{k-1}(P))\\
			&=-(k-1)f_{k-1}(P)^2+(3k^2-3k-2)f_{k-1}(P)-2k(k+1)(k-2),		
		\end{split}
	\end{equation*}
	where $(*)$ follows from Barnette's lower bound theorem for the number of vertices $f_0(P)$ \cite{Barnette71}. 
    
    Now we add $3f_{k-1}(P)$ to both sides of the inequality,
	\[ 3\tau_2(P)+3f_{k-1}(P)\ \geq\ -(k-1)f_{k-1}(P)^2+(3k^2-3k+1)f_{k-1}(P)-2k(k+1)(k-2). \]
	The right-hand side is a concave quadratic function of $f_{k-1}(P)$. Since as we saw at the beginning $k+1\leq f_{k-1}(P)\leq 2k-1$, the minimum occurs at one of the two endpoints, namely $\min\{(k+1)(k+2),\,k(k+4)\}$. When $k\geq 4$, both are at least $6k+3$, which leads to $\tau_2(P)+f_{k-1}(P)\geq 2k+1$. For $k=3$ we have $f_2(P)\in\{4,5\}$: if $f_2(P)=5$ the above gives $k(k+4)=21=6k+3$, and if $f_2(P)=4$ then $P=\sigma^3$ and $\tau_2(P)=4\geq 3=2k+1-f_2(P)$. This proves the lemma.
\end{proof}

Next, we introduce three numbers, each of which has an explicit combinatorial meaning to be explained later. 
\begin{definition}
	Define $$m_\ell(d):=a\binom{\ell-1}{2}+\binom{r}{2}, \quad \text{where } d=a(\ell-1)+r \text{ and } 0\leq r< \ell-1,$$
	$$\text{and} \quad \alpha_{d, \ell, k}:=\frac{\binom{k}{2}m_\ell(d)}{3\binom{d}{2}}.$$ 
	For $0\leq i<k$, we let $$U_i(d, k):=\binom{d-i}{k-i}\frac{\binom{\lfloor d/2\rfloor}{i}+\binom{\lceil d/2\rceil}{i}}{\binom{\lfloor d/2\rfloor}{k}+\binom{\lceil d/2\rceil}{k}}.$$
\end{definition}
The number $m_\ell(d)$ is the maximum number of edges in a disjoint union of cliques on a total of $d$ vertices, each clique having size $\leq \ell-1$. In particular, $m_2(d)=0$. 
The following theorem due to Nikulin \cite{Nikulin} (see also \cite[Section 5.1]{Khovanskii}) establishes $U_i(d, k)$ as an upper bound on the average number of $i$-faces in a $k$-face of a simple $d$-polytope.

\begin{theorem}[Nikulin] \label{thm: Nikulin}
	Let $P$ be a simple $d$-polytope. For $0\leq i<k\leq \lfloor \frac{d+1}{2}\rfloor$, $$\frac{f_{ik}(P)}{f_k(P)}\leq U_i(d, k).$$
\end{theorem}	

\begin{proposition}\label{prop: UB for triangles}
	Let $\ell \geq 2$. Let $P$ be a simple $d$-polytope. If $P$ has no simplex $\ell$-faces, then $3\tau_2(P)\leq m_\ell(d)f_0(P)$. Consequently, for $2\leq k\leq d$, $\tau_{2k}(P) \leq \alpha_{d, \ell, k} \tau_{0k}(P).$
\end{proposition}
\begin{proof}
	Let $v\in V(P)$, let $u_1, \dots, u_d$ be the neighbors of $v$, and for $S\subseteq [d]$, let $G_S$ be the $|S|$-face containing the vertices $\{v, u_i: i\in S\}$. Define $\Gamma_v=\{S\subseteq [d]: G_S \text{ is a simplex}\}$; since every face of a simplex is a simplex, $\Gamma_v$ is a simplicial complex on vertex set $[d]$.
	
	Let $F\in\Gamma_v$ and let $a\notin F$ be such that $\{x,a\}\in\Gamma_v$ for all $x\in F$. We prove by induction on $|F|$ that $F\cup \{a\}\in\Gamma_v$: Indeed, the case $|F|=1$ is trivial. For $|F|\ge2$ pick any vertex $b\in F$; by induction $(F\setminus b)\cup a\in\Gamma_v$. Now $G_{F\cup a}$ is a simple polytope of dimension $\ge3$ where $G_F$ and $G_{(F\setminus b)\cup a}$ are adjacent simplex facets. By Lemma~\ref{lm: adjacent simplex facet}, $F\cup \{a\}\in\Gamma_v$ .

If $F, F'$ are two maximal faces of $\Gamma_v$ with $x\in F\cap F'$, then for every $y\in F$ and $z\in F'\backslash F$, we claim $\{y,z\}\in\Gamma_v$. If $y=x$, this is immediate since $x,z\in F'\in\Gamma_v$ and $\Gamma_v$ is downward closed; so assume $y\neq x$. In this case $G_{\{x,y\}}$ and $G_{\{x,z\}}$ are adjacent triangular facets of the simple $3$-polytope $G_{\{x,y,z\}}$, so $\{y,z\}\in\Gamma_v$ by Lemma~\ref{lm: adjacent simplex facet}.
    
	% I THINK THIS WAS INCOMPLETE
	% If $F, F'$ are two maximal faces of $\Gamma_v$ with $x\in F\cap F'$, then for every $y\in F$ and $z\in F'\backslash F$, the faces $G_{\{x,y\}}$ and $G_{\{x,z\}}$ are adjacent triangular facets of the simple $3$-polytope $G_{\{x,y,z\}}$, so $\{y,z\}\in\Gamma_v$ by Lemma~\ref{lm: adjacent simplex facet}. 
    
    The above argument implies that $F\cup z\in\Gamma_v$, contradicting that $F$ is a maximal face. Hence, the maximal faces of $\Gamma_v$ are pairwise disjoint. Since $P$ has no simplex $\ell$-faces, every face of $\Gamma_v$ has at most $\ell-1$ vertices. Thus, $f_1(\Gamma_v)\leq m_\ell(d)$ by definition of $m_\ell(d)$. By double counting the triangular 2-faces,
	$$3\tau_2(P)=\sum_{v\in V(P)}\# \{\text{triangular 2-faces containing $v$}\}=\sum_{v\in V(P)} f_1(\Gamma_v)\leq m_\ell(d)f_0(P).$$ This proves the first part. The second part follows directly from (\ref{eq: face_flag}) and simplifying the expression.
\end{proof}

Now we are ready to prove that $f_s(\ell, k)$ is bounded.
\begin{theorem}\label{thm: bounded f_s}
	Let $\ell\geq 2, k\geq 3$, and $d\geq \max\{2k,\ell\}$. Let $P$ be a simple $d$-polytope with no simplex $\ell$-face. If we denote by $c_k(P)$ the number of cube $k$-faces of $P$, then
	$$U_{k-1}(d, k)+\alpha_{d, \ell, k}U_0(d, k)\geq 2k+1-c_k(P)/f_k(P).$$
	Consequently, if $U_{k-1}(d, k)+\alpha_{d, \ell, k}U_0(d, k)< 2k+1$, then every simple polytope of dimension $\geq d$ has either a simplex $\ell$-face or a cube $k$-face.
\end{theorem}
\begin{proof}
	We estimate $f_{k-1, k}(P)=\sum_{\dim F=k} f_{k-1}(F)$ by counting the facets of two types of $k$-faces: A cube $k$-face has exactly $2k$ facets. The other faces satisfy the condition in Lemma \ref{lm: simple polytope property}. Since the cube $C^k$ has no triangular $2$-faces, $\tau_2(F)=0$ whenever $F\cong C^k$, so
$$\sum_{F\ncong C^k}\tau_2(F)=\sum_{\dim F=k}\tau_2(F)=\tau_{2k}(P).$$
Using this identity, together with Theorem \ref{thm: Nikulin} and Lemma \ref{lm: simple polytope property}
\begin{equation*}
	\begin{split}
		U_{k-1}(d, k)f_k(P)\geq f_{k-1,k}(P)&=\sum_{\dim F=k}f_{k-1}(F)\ =\ 2k\,c_k(P)+\sum_{F\ncong C^k}f_{k-1}(F)\\
		&\geq 2k\,c_k(P)+\sum_{F\ncong C^k}\bigl(2k+1-\tau_2(F)\bigr)\\
		&= 2k\,c_k(P)+(2k+1)\bigl(f_k(P)-c_k(P)\bigr)-\sum_{F\ncong C^k}\tau_2(F)\\
		&= 2k\,c_k(P)+(2k+1)\bigl(f_k(P)-c_k(P)\bigr)-\tau_{2k}(P)\\
		&= (2k+1)f_k(P)-c_k(P)-\tau_{2k}(P). 
	\end{split}
\end{equation*}

	On the other hand, Proposition \ref{prop: UB for triangles} and Theorem \ref{thm: Nikulin} imply that $$\tau_{2k}(P)\leq \alpha_{d, \ell, k}\tau_{0k}(P)\leq \alpha_{d, \ell, k}U_0(d, k)f_k(P).$$
	Combining the two inequalities, we get
	$$U_{k-1}(d, k)f_k(P)\geq (2k+1)f_k(P)-c_k(P)-\alpha_{d, \ell, k}U_0(d, k)f_k(P),$$
	$$\text{i.e.}, \quad U_{k-1}(d, k)+\alpha_{d, \ell, k}U_0(d, k)\geq 2k+1-c_k(P)/f_k(P).$$ 
	
	For the last part, assume that $P$ is a simple $d$-polytope with no simplex $\ell$-face and furthermore, $U_{k-1}(d, k)+\alpha_{d, \ell, k}U_0(d, k)< 2k+1$. Then by the proved inequality, $c_k(P)>0$, i.e., $P$ must have a cube $k$-face. For simple polytopes of dimension $>d$, we take any $d$-face (which is a simple $d$-polytope) and apply the same argument.
\end{proof}

\begin{corollary}\label{cor1: f_s(2,k)}
	For $k\geq 3$, $f_s(2, k)\leq 2k^2-1$.
\end{corollary}
\begin{proof}
		Since $m_2(d)=0$, we have $\alpha_{d,2,k}=0$ and the criterion of Theorem~\ref{thm: bounded f_s} is equivalent to $U_{k-1}(d,k)<2k+1$. When $d=2m$, $$U_{k-1}(2m, k)=(2m-k+1)\frac{\binom{m}{k-1}}{\binom{m}{k}}=\frac{k(2m-k+1)}{m-k+1}.$$  Similarly, when $d=2m-1$, $$U_{k-1}(2m-1, k)=(2m-k)\frac{\binom{m-1}{k-1}+\binom{m}{k-1}}{\binom{m-1}{k}+\binom{m}{k}}=(2m-k)\frac{\frac{2m-k+1}{m}\binom{m}{k-1}}{\frac{2m-k}{m}\binom{m}{k}}=\frac{k(2m-k+1)}{m-k+1}.$$
		Hence $U_{k-1}(d,k)<2k+1$ if and only if $k(2m-k+1)<(2k+1)(m-k+1)$, i.e.\ $k^2<m+1$. Since $m=\lceil d/2\rceil$, this leads to $d\geq 2k^2-1$.
\end{proof}

\begin{corollary}\label{cor2: f_s(l,k)}
For $\ell\ge3$ and $k\ge3$,
\[
f_s(\ell,k)\ \le\ 4k^{2}+\tfrac13\,k(k-1)(\ell-2)\,2^{k}.
\]
In particular, $f_s(\ell, k)\le\ \frac{1}{2}k^{2}\ell\,2^{k}$.
\end{corollary}

\begin{proof}
Set $m:=2k^2+\left\lfloor
\tfrac16 k(k-1)(\ell-2)2^k
\right\rfloor$ and $d:=2m$.
Since $k,\ell\geq 3$, $m\geq 2k^2+\frac{4}{3}k(k-1)$. Thus it follows from the computation in Corollary \ref{cor1: f_s(2,k)} that
\[
U_{k-1}(2m,k)=\frac{k(2m-k+1)}{m-k+1}=2k+\frac{k(k-1)}{m-k+1} \leq 2k+\frac{k^2-k}{\frac{10}{3}k^2-\frac{7}{3}k+1}<2k+\frac{k^2-k}{\frac{10}{3}k^2-\frac{10}{3}k}=2k+\frac{3}{10}.
\]

On the other hand,
\begin{equation*}
	\begin{split}
		U_0(2m,k)&=\prod_{i=0}^{k-1}\frac{2m-i}{m-i}=2^{k}\prod_{i=0}^{k-1}\Big(1+\frac{i}{2(m-i)}\Big)\\
		& \le\ 2^{k}\exp\left(\sum_{i=0}^{k-1}\frac{i}{2(m-i)}\right)
		\ <\ 2^{k}\exp\left(\frac{1}{m-k+1}\sum_{i=0}^{k-1}\frac{i}{2}\right)\\
		& \le\ 2^k\exp\left(\frac{k(k-1)}{4(m-k+1)}\right)\leq 2^k\cdot e^{1/8},
	\end{split}
\end{equation*}
where the last step is because $m\ge2k^{2}$ gives $\frac{k(k-1)}{4(m-k+1)}\le\frac{k(k-1)}{4(2k^{2}-k+1)}<\frac18$. 

Writing $d=a(\ell-1)+r$ with $0\leq r<\ell-1$, the bound $\binom{r}{2}\leq\frac{(\ell-2)r}{2}$ gives $m_\ell(d)\leq \frac{(\ell-2)d}{2}$, whence
$$\alpha_{d,\ell,k}=\frac{\binom{k}{2}m_\ell(d)}{3\binom{d}{2}}\leq \frac{k(k-1)(\ell-2)}{6(d-1)}=\frac{k(k-1)(\ell-2)}{6(2m-1)}.$$
Since $2m-1\geq \frac{1}{3}k(k-1)(\ell-2)2^k$, it follows that $$\alpha_{d,\ell,k}U_0(2m, k)\leq \frac{k(k-1)(\ell-2)2^k}{6(2m-1)}\cdot e^{1/8} \leq \frac{e^{1/8}}{2}<\frac{4}{7}.$$
Putting everything together,
$$U_{k-1}(d,k)+\alpha_{d,\ell,k}U_0(d,k)<2k+\tfrac3{10}+\tfrac47<2k+1.$$

So by Theorem \ref{thm: bounded f_s}, $f_s(\ell,k)\le d=2m$. Finally, by using $k, \ell\geq 3$, we obtain that $4k^{2}\le\tfrac16k^{2}\ell2^{k}$ and
$\tfrac13k(k-1)(\ell-2)2^{k}\le\tfrac13k^{2}\ell2^{k}$, and thus $d\le \frac{1}{2}k^{2}\ell2^{k}$.
\end{proof}
%\begin{corollary} For $\ell\geq 2$ and $k\geq 3$, $$f_s(\ell,k)\ \leq\ 4k^2+2k(k-1)(\ell-2)\,3^{k-1}.$$ In particular, $f_s(\ell,k)\leq 4k^2\ell\cdot 3^{k-1}$.
%\end{corollary}
%\begin{proof}
%	Set $m=2k^2+k(k-1)(\ell-2)3^{k-1}$ and $d=2m$. Since $m\geq 2k^2$, direct computation gives $m-k+1>2k(k-1)$, so $$U_{k-1}(d,k)=2k+\frac{k(k-1)}{m-k+1}<2k+\tfrac12.$$ Likewise $m\geq 2k^2$ ensures $i\leq k-1\leq m/2$ for all $0\leq i\leq k-1$, so each factor of $U_0(d,k)=\prod_{i=0}^{k-1}\bigl(2+\tfrac{i}{m-i}\bigr)$ is at most $3$, giving $U_0(d,k)\leq 3^k$.
	
%	Writing $d=a(\ell-1)+r$ with $0\leq r<\ell-1$, the bound $\binom{r}{2}\leq\frac{(\ell-2)r}{2}$ gives $m_\ell(d)\leq \frac{(\ell-2)d}{2}$, whence $$\alpha_{d,\ell,k}=\frac{\binom{k}{2}m_\ell(d)}{3\binom{d}{2}}\leq \frac{k(k-1)(\ell-2)}{6(d-1)}=\frac{k(k-1)(\ell-2)}{6(2m-1)}.$$ By the choice of $m$, $2m-1>k(k-1)(\ell-2)3^{k-1}$, so $\alpha_{d,\ell,k}U_0(d,k)<\tfrac12$. Combining, $$U_{k-1}(d,k)+\alpha_{d,\ell,k}U_0(d,k)\ <\ \left(2k+\tfrac12\right)+\tfrac12\ =\ 2k+1,$$ so by Theorem~\ref{thm: bounded f_s}, $f_s(\ell,k)\leq d=2m$. The stated form follows directly. The "in particular" bound follows from $(k-1)(\ell-2)\leq k\ell$ and $4k^2\leq 2k^2\ell\cdot 3^{k-1}$ (as $\ell\geq2,\,3^{k-1}\geq1$).
%\end{proof}

\begin{remark}
	  The upper bound of $f_s(\ell, k)$ given by Theorem \ref{thm: bounded f_s} is usually not optimal. In the case $\ell \geq 3$ and $k=4$, the inequality in Lemma~\ref{lm: simple polytope property} can be strengthened to $f_3(P)\geq 9-\tau_2(P)/2$. Indeed, by Blind--Blind only the range $5\leq f_3(P)\leq 8$ has to be checked, and the classification of simple $4$-polytopes with at most $8$ facets shows that $\tau_2(P)\geq 6$ when $f_3(P)\leq 6$, and $\tau_2(P)\geq 4$ when $f_3(P)\in\{7,8\}$; equality holds for $\sigma^2\times\sigma^2$ and $\sigma^2\times C^2$ \cite{GrunbaumSreedharan, Grunbaum}. Applying the argument of Theorem~\ref{thm: bounded f_s} with this inequality, a simple polytope of dimension $\geq d$ has a simplex $\ell$-face or a cube $4$-face provided that $U_3(d,4)+\frac12\alpha_{d,\ell,4}U_0(d,4)<9$. For $\ell=3$ the criterion improves $f_s(3,4)\leq 63$ to $f_s(3,4)\leq 47$.
\end{remark}
Finally, Theorem~\ref{thm: bounded f_s} says nothing about $f_s(\ell,2)$. In this case, the boundedness of $f_s(\ell, 2)$ follows from the monotonicity of $f_s(\ell,k)$ in $k$.
\begin{proposition}
	For all $\ell\geq 2$ and $k\geq 3$ we have $f_s(\ell,k)\geq f_s(\ell,k-1)+1$.
\end{proposition}
\begin{proof}
	Let $d=f_s(\ell,k-1)$. By minimality of $d$, there is a simple polytope $P$ of dimension exactly $d-1$ with no simplex $\ell$-face and no cube $(k-1)$-face. By Lemma~\ref{lem:product}, $P\times\sigma^1$ is a simple $d$-polytope with $s(P\times\sigma^1)=\max\{s(P),1\}\leq \ell-1$ and $c(P\times\sigma^1)=c(P)+1\leq k-1$, hence $f_s(\ell,k)\geq d+1$.
\end{proof}

\section{Lower bounds on $f(\ell, k)$ and $f_s(\ell, k)$: Proof of Theorem B} \label{sec:lowerbound}

	 Our starting point is a remark of Kalai \cite[Chapter 19]{Kal17}, which we quote with the notation adapted to ours: ``Julian Pfeifle showed on the basis of the Wythoff construction (see Chapter 18), that $f(k, k) > (2k - 1)(k - 1)+1$, for $k \geq 3$.'' Our constructions are inspired by this observation, but yield stronger bounds and cover the full two-parameter range; see Theorem \ref{thm:general} and Theorem \ref{thm:simple}. Our arguments use three ingredients: two basic operations that build higher-dimensional polytopes out of lower-dimensional ones, the product and the join operations, together with Wythoff's construction. The definitions and the properties of the product and the join are given below; see also \cite{HRZ}. Wythoff's construction is recalled in Remark \ref{rm:wythoff}.

	\begin{definition}
		Let $P_1\subseteq\R^{d_1}$ be a $d_1$-polytope and $P_2\subseteq\R^{d_2}$ a $d_2$-polytope. The product of $P_1$ and $P_2$ is defined as $P_1\times P_2:=\{(x,y): x\in P_1,\ y\in P_2\}\subseteq \R^{d_1+d_2}$.
	\end{definition}
	\begin{lemma}\label{lem:product}
		Let $P_1, P_2$ be polytopes and let $P=P_1\times P_2$.
		\begin{enumerate}
			\item $\dim P=\dim P_1+\dim P_2$. Every nonempty face of $P$ is of the form $F_1\times F_2$, where each $F_i$ is a nonempty face of $P_i$.
			\item $P$ is simple whenever $P_1$ and $P_2$ are.
			\item $s(P)=\max\{s(P_1), s(P_2)\}$ and $c(P)= c(P_1)+c(P_2)$.
		\end{enumerate}
	\end{lemma}
	
	\begin{definition}\label{def:join}
		Let $P_1$ be a $d_1$-polytope and $P_2$ a $d_2$-polytope. Embed them in skew affine subspaces of $\R^{d_1+d_2+1}$:
		\[ P'_1=P_1\times\{0\}\times\{0\},\qquad P_2'=\{0\}\times P_2\times\{1\}. \]
		Finally, set the \emph{join} of $P_1$ and $P_2$ as $P_1*P_2:=\conv(P_1'\cup P_2')$.
	\end{definition}
	
	\begin{lemma}\label{lem:join}
		Let $P_1, P_2$ be polytopes.
		\begin{enumerate}
			\item $P_1*P_2$ is a polytope of dimension $\dim P_1+\dim P_2+1$. Its faces are of the form $F_1*F_2$, where each $F_i$ is a  face of $P_i$ (including the empty face and $P_i$ itself).
			\item $P*\sigma^{p-1}$ is the $p$-fold pyramid over $P$.
			\item $s(P_1*P_2)=s(P_1)+s(P_2)+1$, and $c(P_1*P_2)=\max\{c(P_1), c(P_2)\}$ provided $\dim P_i\geq 1$ for some $i$.
		\end{enumerate}
	\end{lemma}
	
	Now we are ready to give constructions that lead to lower bounds on $f(\ell, k)$ and $f_s(\ell, k)$.
	\begin{theorem}\label{thm:general}
		For all $\ell\ge2$ and $k\ge2$,
		\[ f(\ell,k) \ \ge\ D(\ell)\,(k-1)+1, \quad \text{where}\quad D(\ell):=5\left\lfloor \ell/2\right\rfloor+(\ell\bmod 2)-1. \]
	\end{theorem}
	\begin{proof}
		Put $r=\lfloor \ell/2\rfloor$ and $p=\ell-2r\in\{0,1\}$, let
		\[ B_\ell\ :=\ \underbrace{\Sigma*\cdots*\Sigma}_{r}\,*\,\sigma^{p-1},\]
		and take $P$ as the product of $k-1$ copies of $B_\ell$. Then by Lemmas \ref{lem:product} and \ref{lem:join}, $$\dim P=\dim B_\ell\cdot (k-1)=(5r+p-1)(k-1)=D(\ell)(k-1),$$ $$s(P)=s(B_\ell)=r\,s(\Sigma)+s(\sigma^{p-1})+r=2r+p-1<\ell,$$
		$$c(P)=c(B_\ell)(k-1)=\max\{c(\Sigma), c(\sigma^{p-1})\}\cdot (k-1)=k-1<k.$$
		Hence $P$ has no simplex $\ell$-face and no cube $k$-face and the lower bound follows.
	\end{proof}

Now we move on to giving a lower bound on $f_s(\ell,k)$. The building block is the orbit polytope of a point with three distinct coordinate values; its relevant properties become transparent once one
observes that it is a hyperplane section of a cube.

\begin{definition}\label{def:atom}
For integers $p,q\ge1$, let $N:=p+q+1$ and fix $0<\theta<1$. Let
\[
T(p,q):=\conv\Big\{\pi\cdot\big(\underbrace{1,\dots,1}_{p},\,\theta,\,
\underbrace{0,\dots,0}_{q}\big)\ :\ \pi\in S_N\Big\}\ \subseteq\ \R^{N},
\]
the orbit polytope, under the coordinate permutations, of a point with three distinct coordinate
values of multiplicities $p,1,q$.
\end{definition}

\begin{lemma}\label{lm:slice}
Let $C=[0,1]^N$ and $H=\{x\in\R^N:\ x_1+\dots+x_N=p+\theta\}$. Then
\[
T(p,q)\ =\ C\cap H .
\]
In particular, $T(p,q)$ is an $(p+q)$-polytope with $\binom{N}{p}(N-p)$ vertices, and its
combinatorial type does not depend on $\theta$.
\end{lemma}

\begin{proof}
Since $0<p+\theta<N$, the hyperplane $H$ meets the interior of $C$, so $P:=C\cap H$ is a polytope of
dimension $N-1=p+q$. Each face of $P$ is $F\cap H$ for some face $F$ of $C$. Since $\theta\notin\{0,1\}$, no vertex of $C$ lies on $H$. Thus every vertex of $P$ is $F\cap H$, where $F$ is an edge of $C$.

Write $e_1,\dots,e_N$ for the standard basis of $\R^N$ and $\supp(v):=\{j:v_j=1\}$ for a $0/1$ vector $v$. Every edge of $C$ has the form $[\,v,\ v+e_i\,]$ with $v$ a $0/1$ vector and $i\notin\supp(v)$. The coordinate sums of its two endpoints are $|\supp(v)|$ and $|\supp(v)|+1$, so this edge meets $H$ if and only if $|\supp(v)|<p+\theta<|\supp(v)|+1$, that is, if and only if $|\supp(v)|=p$. In this case, the intersection point is $v+\theta e_i$. These are precisely the points listed in Definition~\ref{def:atom}, and their number is $\binom{N}{p}(N-p)$: choose $\supp(v)$, then $i\notin\supp(v)$. Hence $P$ and $T(p,q)$ have the same
vertex set, so they are equal.
\end{proof}

\begin{lemma}\label{lm:atom}
$T(p,q)$ is a simple $(p+q)$-polytope with
\[
s\big(T(p,q)\big)=\max\{p,q\},\qquad c\big(T(p,q)\big)=1 .
\]
\end{lemma}

\begin{proof}
We use the description $T(p,q)=C\cap H$ of Lemma~\ref{lm:slice}.

Every face of the cube $C$ is of the form $F(A,B)=\{x\in C:\ x_i=1\ (i\in A),\ x_i=0\ (i\in B)\}$ for disjoint $A,B\subseteq[N]$. Set $m:=N-|A|-|B|$ and $M:=[N]\setminus(A\cup B)$. Thus every face of $T(p,q)$ is
\begin{equation}\label{eq:face}
F(A,B)\cap H=\Big\{x\in[0,1]^{M}:\ \sum_{i\in M}x_i=r+\theta\Big\},\qquad r:=p-|A|.
\end{equation}
By Lemma~\ref{lm:slice}, $F(A, B)\cap H$ is nonempty exactly when $0\le r\le m-1$, in which case it is a
polytope of dimension $m-1$ with $\binom{m}{r}(m-r)$ vertices.

A face \eqref{eq:face} is a simplex if and only if $\binom{m}{r}(m-r)=m$. This holds for $r=0$ and for $r=m-1$,
while for $1\le r\le m-2$ we have $\binom{m}{r}\ge m$ and $m-r\ge2$, so $\binom mr(m-r)\ge2m>m$.
If $r=0$, that is $|A|=p$, the face has dimension $m-1=q-|B|\le q$, with equality for $B=\emptyset$;
if $r=m-1$, that is $|B|=q$, it has dimension $m-1=p-|A|\le p$, with equality for $A=\emptyset$.
Hence $s(T(p,q))=\max\{p,q\}$.

By \eqref{eq:face}, the $2$-faces of $T(p,q)$ are the sections with $m=3$, and
these have $\binom{3}{r}(3-r)$ vertices, namely $3$ for $r\in\{0,2\}$ and $6$ for $r=1$. So no
$2$-face is a square, and therefore $c(T(p,q))=1$.

Finally, the facets of $C$ are $\{x_j=0\}$ and $\{x_j=1\}$, so every facet of $T(p,q)$ is
of this form intersected with $H$. The vertex $v+\theta e_i$ lies on $\{x_j=1\}$ for the $p$ indices
$j\in\supp(v)$ and on $\{x_j=0\}$ for the $q$ indices $j\notin\supp(v)\cup\{i\}$, and on no other
facet of $C$, because its $i$-th coordinate is $\theta\in(0,1)$. Hence every vertex of $T(p,q)$ lies on
exactly $p+q$ facets, and $T(p,q)$ is simple.

\end{proof}

\begin{remark}\label{rm:wythoff}
The polytopes $T(p,q)$ arise from \emph{Wythoff's construction} \cite{W18,Cox35},
\cite[\S5.7 and \S11.6]{Cox}. Let $W$ be a finite group generated by reflections with Coxeter graph
$\Gamma$ on the node set $S$, each $s\in S$ acting as the reflection in a hyperplane $H_s$, and let
$R\subseteq S$. Choose a point $\lambda_R$ in the closed fundamental chamber of $W$ lying off
exactly the mirrors indexed by $R$, that is,
\[
\lambda_R\in H_s\iff s\notin R ,
\]
and put
\[
t_R(W):=\conv\big(W\lambda_R\big).
\]
In the graph one \emph{circles} the nodes of $R$; the combinatorial type of $t_R(W)$
depends only on $\Gamma$ and $R$, not on the choice of $\lambda_R$.

Take $W=S_N$ acting on $\R^N$ by permuting coordinates. Its Coxeter graph is a path with
$N-1$ nodes $s_0,\dots,s_{N-2}$, where $s_{i-1}$ is the transposition of the $i$-th and
$(i+1)$-st coordinates, with mirror $\{x_i=x_{i+1}\}$. The closed fundamental chamber is
$\{x_1\ge\cdots\ge x_N\}$. The point $(1,\dots,1,\theta,0,\dots,0)$ of Definition~\ref{def:atom} lies in
this chamber and lies off exactly the two mirrors $\{x_p=x_{p+1}\}$ and $\{x_{p+1}=x_{p+2}\}$.
Hence
\[
T(p,q)=t_{\{s_{p-1},s_p\}}(S_N)=t_{p-1,p}\{3^{N-2}\}.
\]
The second expression is Coxeter's symbol for the uniform polytopes obtained from the regular
simplex $\{3^{N-2}\}$ by ringing the nodes with the indicated indices.
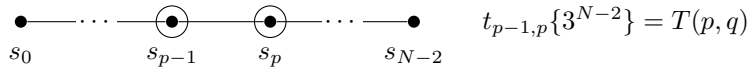
\begin{figure}[H]
\centering
\begin{tikzpicture}[x=1cm,y=1cm,every node/.style={font=\small}]
\tikzset{nd/.style={circle,fill=black,inner sep=1.6pt},
         rng/.style={draw,circle,inner sep=4.2pt}}
  \node[nd] (c0) at (0,0) {};
  \node[nd] (c1) at (2.0,0) {}; \node[rng] at (2.0,0) {};
  \node[nd] (c2) at (3.3,0) {}; \node[rng] at (3.3,0) {};
  \node[nd] (c3) at (5.2,0) {};
  \draw (c0)--(c1); \draw (c1)--(c2); \draw (c2)--(c3);
  \node[fill=white,inner sep=1pt] at (1.0,0) {$\cdots$};
  \node[fill=white,inner sep=1pt] at (4.25,0) {$\cdots$};
  \node[below=6pt] at (c0) {$s_0$};      \node[below=6pt] at (c1) {$s_{p-1}$};
  \node[below=6pt] at (c2) {$s_{p}$};    \node[below=6pt] at (c3) {$s_{N-2}$};
  \node[right=22pt,anchor=west] at (c3) {$t_{p-1,p}\{3^{N-2}\}=T(p,q)$};
\end{tikzpicture}
\caption{The Coxeter graph of type $A_{N-1}$ with the two adjacent ringed nodes $s_{p-1},s_p$.}
\label{fig:diagrams}
\end{figure}
\end{remark}

\begin{theorem}\label{thm:simple}
For all $\ell,k\ge2$,
\[
f_s(\ell,k)\ \ge\ \max\{4,\ 2(\ell-1)\}\,(k-1)+1 .
\]
\end{theorem}

\begin{proof}
For $\ell\le3$ let $P$ be the product of $k-1$ copies of the $120$-cell $\Sigma$; by Lemma~\ref{lem:product}, $P$ is
a simple $4(k-1)$-polytope with $s(P)=s(\Sigma)=1<\ell$ and $c(P)=(k-1)c(\Sigma)=k-1<k$.

For $\ell\ge3$ let $Q:=T(\ell-1,\ell-1)$, a simple $2(\ell-1)$-polytope with $s(Q)=\ell-1$ and
$c(Q)=1$ by Lemma~\ref{lm:atom}, and let $P$ be the product of $k-1$ copies of $Q$. By Lemma~\ref{lem:product},
$P$ is simple of dimension $2(\ell-1)(k-1)$, with $s(P)=s(Q)=\ell-1<\ell$ and
$c(P)=(k-1)c(Q)=k-1<k$. In either case $P$ has no simplex $\ell$-face and no cube $k$-face.
\end{proof}

\begin{remark}
On the diagonal, Theorem~\ref{thm:simple} gives $f_s(k,k)\ge2(k-1)^2+1$, which is asymptotically
$2k^2$.
\end{remark}

\begin{remark}\label{rm:new22}
Let $$g(t):=\max\{d:\ \text{there is a square-free simple $d$-polytope $P$ with }s(P)\le t\}.$$
Lemma~\ref{lm:atom} applied to $T(t,t)$ gives
\[
g(t)\ \ge\ \max\{4,\,2t\}\qquad(t\ge1),
\]
the two equal multiplicities giving the largest dimension for a prescribed value of $s$. In
particular $g(3)\ge6$, witnessed by $T(3,3)$, a simple $6$-polytope with $140$ vertices and $14$
facets, and consequently $f_s(4,k)\ge6(k-1)+1$.

Two questions remain. First, can we further improve the lower bound for $g(t)$? Among the Wythoffians we found no construction better than $T(p,q)$. Second, is there an operation on polytopes raising the dimension without raising $s$ or $c$ proportionally? The product and the join each pay in a fixed currency, and Theorem~\ref{thm:simple} exhausts what the product alone can do with the blocks we know.
\end{remark}

\section{$3$-faces that appear in high dimensions: Proof of Theorem C}
\label{sec:linearprogramming}

In this section we prove our third main result. Informally, it says that
once a polytope is high-dimensional enough, it cannot avoid containing a
combinatorially simple three-dimensional piece: some $3$-face with only a
handful of facets. Our argument follows closely the linear programming approach of Meisinger, Kleinschmidt, and Kalai \cite{Meisinger2000}. The rationality
assumption in their theorem was needed because their proof used the
nonnegativity of the toric $g_3$-numbers, which at the time was known only
for rational polytopes via intersection cohomology. Karu's subsequent hard
Lefschetz theorem extended this nonnegativity to all convex polytopes
\cite{Karu2004}.  Thus, our result holds for arbitrary convex polytopes.

\begin{theorem}\label{thm: d15 cap13}
Every convex polytope of dimension $d\geq 15$ has a $3$-face with at most
$13$ facets.
\end{theorem}

It is enough to prove the theorem when $d=15$. Indeed, if $d>15$, take any $15$-face $Q$ of the given polytope. Every $3$-face of $Q$ is also a $3$-face of the original polytope. We therefore assume throughout this section that $P$ is a $15$-dimensional polytope such that
\begin{equation}\label{eq: d15 counterexample}
 f_2(F)\geq 14\qquad\text{for every $3$-face $F$ of $P$},
\end{equation}
and derive a contradiction.

\subsection{Flag functionals and convolution}
\label{subsec: flag functionals and convolution}

We recall the definition of flag $f$-numbers and flag $h$-numbers of $P$ (see \cite{billerabjorner-facesurvey} and the references there for an overview). These numbers are truly important for understanding the combinatorics of faces.

\begin{definition}\label{def:flag-f}
Let $P$ be an $n$-polytope. For a subset $S=\{s_1<s_2<\dots<s_k\}\subseteq\{0,1,\dots,n-1\}$, the
\emph{flag number} $f_S(P)$ is the number of chains of proper faces
\[
  F_1\subset F_2\subset\cdots\subset F_k, \qquad \dim F_j=s_j \ \ (1\le j\le k).
\]
By convention $f_\emptyset(P)=1$. In particular, $f_{\{i\}}(P)$ is the usual face number $f_i(P)$ and $f_{\{i, j\}}(P)$ equals $f_{ij}(P)$ we defined in previous sections; we continue using the notions $f_i(P)$ and $f_{ij}(P)$ for simplicity. The vector $\bigl(f_S(P)\bigr)_{S}$ is the \emph{flag $f$-vector}.
\end{definition}

Informally, $f_S(P)$ counts incidences between faces of the dimensions listed in $S$; for instance $f_{\{0,2\}}(P)$ counts pairs consisting of a vertex and a $2$-face containing it.

\begin{definition}\label{def:flag-h}
The \emph{flag $h$-numbers} of $P$ are defined from the flag numbers by inclusion--exclusion:
\[
  h_S(P) \;=\; \sum_{T\subseteq S} (-1)^{\,|S\setminus T|}\, f_T(P),
  \qquad S\subseteq\{0,1,\dots,n-1\},
\]
or equivalently, $f_S(P)=\sum_{T\subseteq S} h_T(P)$.
\end{definition}

We are particularly interested in two families of numbers that arise from the face numbers. The first family consists of the toric $g$-numbers. They are defined by the following recursion.

\begin{definition}\label{def:toric}
For every polytope $P$, define polynomials $f(P,x),\,g(P,x)\in\mathbb{Z}[x]$
by:
\begin{enumerate}
  \item $f(\emptyset,x)=g(\emptyset,x)=1$;
  \item if $\dim P=n\ge 0$, then
  \[
    f(P,x) \;=\; \sum_{\emptyset\,\le\, G\,<\,P} g(G,x)\,(x-1)^{\,n-1-\dim G};
  \]
   \item writing $f(P,x)=\sum_{i=0}^{n} h_i(P)\,x^{\,i}$,
  \[
    g(P,x) \;=\; \sum_{i=0}^{\lfloor n/2\rfloor}\bigl(h_i(P)-h_{i-1}(P)\bigr)\,x^{i},
    \qquad h_{-1}(P):=0 .
  \]
\end{enumerate}
The \emph{toric $h$-vector} of $P$ is $h(P)=\bigl(h_0(P),\dots,h_n(P)\bigr)$ and the \emph{toric $g$-vector} is $g(P)=\bigl(g_0(P),\dots,g_{\lfloor n/2\rfloor}(P)\bigr)$, where
\[
  g_0(P)=h_0(P)=1,\qquad
  g_k(P)=h_k(P)-h_{k-1}(P)\quad\text{for } 1\le k\le\lfloor\tfrac{n}{2}\rfloor .
\]
\end{definition}

When $P$ is simplicial, $h(P)$ and $g(P)$ recover the classical simplicial $h$- and $g$-vectors, and the toric $g$-inequalities recalled below (Theorem~\ref{thm: toric g inequalities}) specialize to the $g$-theorem inequalities of Billera--Lee and Stanley. The point of the toric
generalization is that it extends these inequalities to \emph{every} convex polytope, simplicial or not.

Next we recall the definition of the $cd$-index of $P$ (see \cite{bayer-cd-survey} for an excellent survey on the topic).

\begin{definition}\label{def:ab}
Let $a,b$ be noncommuting indeterminates. For $S\subseteq\{0,1,\dots,n-1\}$ let $u_S=u_0u_1\cdots u_{n-1}$ be the word of length $n$ with
\[
  u_i=\begin{cases} b, & i\in S,\\ a, & i\notin S.\end{cases}
\]
The \emph{$ab$-index} of $P$ is the noncommutative homogeneous polynomial of degree $n$
\[
  \Psi_P(a,b) \;=\; \sum_{S\subseteq\{0,1,\dots,n-1\}} h_S(P)\, u_S .
\]
\end{definition}

By a theorem of Fine (unpublished, see \cite{BayerKlapper1991}), setting $c=a+b$ and
$d=ab+ba$, there is a unique polynomial $\phi_P$ in the noncommuting variables $c$ and $d$ such that $\Psi_P(a,b)=\phi_P(c,d)$. Moreover, $\phi_P$ is homogeneous of degree $n$ when $\deg c=1$ and $\deg d=2$.
\begin{definition}
The polynomial $\phi_P(c,d)$ that satisfies $\phi_P(c,d)=\Psi_P(a,b)$ is the \emph{$cd$-index} of $P$. Its monomials are the $cd$-words of degree $n$.
\end{definition}

It follows from a result of Bayer and Billera \cite{BayerBillera1985} that the number of $cd$-words of an $n$-polytope is exactly the $(n+1)$-th Fibonacci number $F_{n+1}$, where $F_1=F_2=1$ and $F_k=F_{k-1}+F_{k-2}$ for
$k\geq 3$. This number is also the dimension of the linear span of the flag vectors of $n$-polytopes.

In what follows, we consider a \emph{flag functional} of $P$, which is a rational linear combination of the flag numbers. By definition, the flag numbers can be encoded by the $cd$-index: each coefficient of the
$cd$-index is a linear combination of flag numbers. When $P$ is a
$15$-polytope, there are $F_{16}=987$ such words. The point of this representation is that it puts all the
inequalities used in the linear-programming certificate into a common
coordinate system. This idea goes back to \cite{Kalai}, and was made more
general and systematic in \cite{MKK-2000,Meisinger2000}. We adopt their
method and begin with two standard families of flag functionals that are
known to be nonnegative on every convex polytope. The first family is
known as the toric $g$-inequalities; it was first proved by Stanley for rational
polytopes using intersection cohomology \cite{Stanley1987} and then for all polytopes using Karu's hard Lefschetz theorem \cite{Karu2004}.

\begin{theorem}\label{thm: toric g inequalities}
    Let $P$ be an $n$-polytope. Then $g_i(P)\geq 0$ for all $0\leq i\leq\left\lfloor\frac n2\right\rfloor$.
\end{theorem}

The second family consists of $cd$-index inequalities. For an $n$-polytope
$P$, write its $cd$-index $\phi_P(c,d)=\sum_w\phi_w(P)w$, where $w$ ranges
over all $cd$-words. For example, since $h_\emptyset=1$ and only $c^n$ contributes to $u_{\emptyset}=a^n$, $\phi_{c^n}=1$. Also $h_{\{0\}}=f_0-1$ and only $c^n, dc^{n-2}$ contribute to $u_{\{0\}}=ba^{n-1}$, and thus $\phi_{dc^{n-2}}=f_0-2$. Likewise one can show that $\phi_{c^{n-2}d}=f_{n-1}-2$.  Billera and Ehrenborg
\cite[Theorem 5.3]{BilleraEhrenborg2000} give the following coefficientwise
simplex-minimality theorem.

\begin{theorem}\label{thm: cd minimality}
Let $P$ be an $n$-polytope. Then $\phi_w(P)-\phi_w(\sigma^n)\geq 0$.
\end{theorem}

Finally, to obtain flag inequalities of higher dimensional polytopes based
on those among lower dimensional polytopes, we apply Kalai's convolution
construction \cite{Kalai1988}.

\begin{definition}
Let $u$ be a flag functional on $r$-polytopes and let $v$ be a flag
functional on $s$-polytopes. Their convolution is the flag functional on
$(r+s+1)$-polytopes defined by
\begin{equation}\label{eq: d15 convolution}
 (u*v)(P)=\sum_{\substack{F\in \mathcal{F}_r(P)}}u(F)v(P/F),
\end{equation}
where $P$ is an $(r+s+1)$-polytope and $P/F$ is the quotient polytope of
dimension $s$.
\end{definition}

In particular, if $u$ and $v$ are nonnegative, then so is $u*v$. On the
flag basis,
\begin{equation}\label{eq: d15 flag convolution}
 f_S*f_T=f_{S\cup\{r\}\cup(r+1+T)},
 \qquad r+1+T:=\{r+1+t:t\in T\}.
\end{equation}
See \cite{MKK-2000} for details on how to get a large collection of linear inequalities for flag numbers of polytopes (plus some others that are known independently).

\subsection{Our proof requires the answer to a linear program, but not its computation.}
\label{subsec: why an lp}

The proof strategy is as follows. For a polytope $P$, we generate flag
functionals that $P$ must satisfy using iterated convolution of toric
$g$-inequalities, the $cd$-index simplex-minimality inequalities, and their duals; we call these \emph{unconditional rows}. In addition, we also use \emph{conditional rows}, whose nonnegativity depends on the hypothetical counterexample $P$. They are of the form

\begin{equation}\label{eq: d15 left anchor}
 (H*G)(P)
 =
 \sum_{F\in \mathcal{F}_r(P)}H(F)G(P/F)
 \geq 0,
\end{equation}
where $H$ is a flag functional derived from condition
\eqref{eq: d15 counterexample}, and $G$ is unconditional. 
We call $H$ the
\emph{left anchor} of the conditional row. The order is essential because
the counterexample hypothesis is inherited by faces, but not by quotients.

Next, suppose $P$ is really a $15$-polytope satisfying \eqref{eq: d15 counterexample},
and let $x$ be its flag vector, written in $cd$-coordinates (equivalently, working modulo the generalized Dehn--Sommerville relations, which is what
lets us record a flag vector by only $987$ numbers instead of $2^{15}$).

Every flag functional $L$ we can prove is nonnegative on $P$ gives one linear inequality $L(x)\geq 0$ that $x$ must satisfy. This is true regardless of whether $L$ is
unconditionally true of all polytopes (Theorems~\ref{thm: toric g
inequalities} and~\ref{thm: cd minimality}, and their convolutions), or
true just for $P$, which is assumed to be a counterexample (the facet and vertex lower bounds established in
Section~\ref{subsec: counterexample inequalities}, and their convolutions). We also note that the single equality $\mathbf 1(x)=1$ states that $P$ is a genuine $15$-polytope. 

Now suppose that we exhibit finitely many such functionals $L_1,\dots,L_N$ and
nonnegative rationals $\lambda_1,\dots,\lambda_N$ satisfying, as an
identity of functionals (true for \emph{every} $cd$-coordinate vector $x$,
not just the one coming from $P$),
\begin{equation}\label{eq: why lp identity}
 \lambda_1L_1+\cdots+\lambda_NL_N=-\mathbf 1 .
\end{equation}
Evaluating \eqref{eq: why lp identity} at $P$'s flag vector, the left-hand
side is a nonnegative combination of nonnegative numbers, so it is
$\geq 0$; but the right-hand side is $-1$. This contradiction is the
core of the proof of Theorem~\ref{thm: d15 cap13}. We will finish once we find (and verify) an identity of the form
\eqref{eq: why lp identity}.

Finding such an identity is a linear feasibility question, thus it can be answered by solving a linear program: Collect a
large finite pool of candidate functionals $L$ built by convolving the
basic inequalities above in every admissible way, up to total grade $16$, place them as the rows of a matrix $A$ in $cd$-coordinates, and ask for
nonnegative $\lambda$ with $\lambda^{\mathsf T}A=-\mathbf 1$. 
By linear programming duality, such a $\lambda$ exists exactly when the system
$x_{c^{15}}=1,\ Ax\geq 0$ has \emph{no} solution $x$, i.e.,\ when no
candidate counterexample's flag vector can satisfy every inequality in the pool at once. This is a consequence of the duality theory of linear programming \cite{schrijver98}. We can test this by minimizing an auxiliary variable $s\geq 0$ subject to
$x_{c^{15}}=1$ and $Ax+s\mathbf 1\geq 0$: a strictly positive minimum certifies infeasibility (and its dual linear program solution supplies a candidate
$\lambda$), while a minimum of zero means the current pool of inequalities is not yet strong enough and prompts a search for further inequalities to add. 

The search process for the right pool of inequalities was iterated over several dimensions and thresholds until a positive minimum was reached.
 Ultimately, the search is a computer linear programming heuristic for \emph{finding} the $L_i$ and $\lambda_i$'s used in identity \eqref{eq: why lp identity}, but we only use the final verified output of the computational search, the actual search is not part of our proof.
 The search will be described in more detail later.
 
\subsection{Inequalities forced by a possible counterexample}
\label{subsec: counterexample inequalities}

Following \cite[Section~2.1]{MKK-2000}, the counterexample hypothesis yields two families of conditional 
inequalities, which come, respectively, from lower bounds on the number of facets and vertices of faces of $P$. In general, for fixed $m\geq3$, define
\[
 \varphi_k(m):=\min\{f_{k-1}(P): P\text{ is an }k\text{-polytope all of whose
 }3\text{-faces have at least }m\text{ facets}\},\]
 \[\psi_k(m):=\min\{f_0(P): P\text{ is an }k\text{-polytope all of whose
 }3\text{-faces have at least $m$ facets}\}.
 \]

In what follows, we establish the facet bounds for $m=14$ and then turn to the vertex bounds for $m=14$. We are following closely the method of \cite[Section~2.1]{MKK-2000}, but
we stress that our success in improving the results of \cite{MKK-2000} is precisely based on obtaining stronger convolution building blocks from the vertex bounds and facet bounds presented below. Computing power alone would not have been enough.

\begin{remark}
The vertex and facet bounds below are stated in the simplest form sufficient for the proof of Theorem~C. In fact, the authors have obtained sharper lower bounds on $\varphi_k(m)$ and $\psi_k(m)$ for all $m,k\geq 4$, but the improvement is not large enough to change the resulting threshold. These bounds are nevertheless of independent interest, and we expect that a sufficiently strong lower bound on $\varphi_k(m)$ or $\psi_k(m)$ would yield a corresponding improvement of Theorem~C.
\end{remark}
\subsubsection{Facet bounds} \label{facetbounds}

The assumption that a counterexample exists yields
\begin{equation}\label{eq: d15 CAP}
 f_2(F)-14\geq 0
\end{equation}
on every $3$-face $F$ of $P$.

\begin{lemma}\label{lm: d15 four face}
Every $4$-face $Q$ of $P$ has at least $15$ facets.
\end{lemma}

\begin{proof}
Choose a facet $F$ of $Q$. By \eqref{eq: d15 counterexample}, $F$ has at
least $14$ facets. Each facet of $F$ is a ridge of $Q$, and hence lies in exactly one further facet of $Q$. These further facets are distinct: if one
of them contained two distinct facets of $F$, then its intersection with
$F$ would be a proper face of $F$ containing two facets of $F$. Together
with $F$, they give at least $15$ facets of $Q$.
\end{proof}

Thus, every $4$-face $Q$ of $P$ satisfies
\begin{equation}\label{eq: d15 K4}
 f_3(Q)-15\geq 0.
\end{equation}
The next lemma is the main geometric input for the facet bounds in higher
dimensions.

\begin{lemma}\label{lm: d15 five face}
Every $5$-face of $P$ has at least $24$ facets.
\end{lemma}

\begin{proof}
Let $Q$ be a $5$-face of $P$, let $R=Q^*$, and put
$n=f_0(R)=f_4(Q)$. A $3$-face of $Q$ is dual to an edge of $R$, and the
facets of that $3$-face correspond to the $2$-faces of $R$ containing the
edge. Summing \eqref{eq: d15 counterexample} over all edges of $R$ gives
\begin{equation}\label{eq: d15 f12 lower}
 f_{\{1,2\}}(R)\geq 14f_1(R).
\end{equation}
Every vertex of $R$ is dual to a $4$-dimensional facet of $Q$, so
Lemma~\ref{lm: d15 four face} shows that its degree is at least $15$.
Consequently,
\begin{equation}\label{eq: d15 edge lower}
 f_1(R)\geq \left\lceil\frac{15n}{2}\right\rceil.
\end{equation}
On the other hand, the flag upper bound theorem
\cite[Theorem 6.5 and Corollary 6.6]{BilleraEhrenborg2000} gives
\begin{equation}\label{eq: d15 f12 upper}
 f_{\{1,2\}}(R)\leq f_{\{1,2\}}(C(n,5))=6(n^2-6n+10),
\end{equation}
where $C(n,5)$ is the cyclic $5$-polytope with $n$ vertices. For the last
equality, recall that $C(n,5)$ is simplicial,
$f_2(C(n,5))=2(n^2-6n+10)$, and each triangular $2$-face contributes three
edge--face incidences.

Lemma~\ref{lm: d15 four face}, followed by the ridge argument in its proof,
first gives $n\geq 16$. Combining
\eqref{eq: d15 f12 lower}--\eqref{eq: d15 f12 upper}, we obtain
\begin{equation}\label{eq: d15 n test}
 14\left\lceil\frac{15n}{2}\right\rceil\leq 6(n^2-6n+10).
\end{equation}
For $16\leq n\leq 23$, the left-hand side minus the right-hand side in
\eqref{eq: d15 n test} is
\[
 \begin{cases}
 141n-6n^2-60,&n\text{ even},\\
 141n-6n^2-53,&n\text{ odd}.
 \end{cases}
\]
Both expressions decrease in this interval. Their values in the highest
even and odd integers are, respectively, $138$ at $n=22$ and $16$ at
$n=23$. Thus \eqref{eq: d15 n test} fails throughout the interval, and
$n\geq 24$.
\end{proof}

\begin{corollary}\label{cor: d15 KUBT}
Every $r$-face of $P$, where $5\leq r\leq 15$, has at least $r+19$ facets.
\end{corollary}

\begin{proof}
The case $r=5$ is Lemma~\ref{lm: d15 five face}. If a facet $H$ of an
$r$-polytope has $m$ facets, then its $m$ ridges lie in $m$ distinct facets
other than $H$. The ambient polytope therefore has at least $m+1$ facets,
and induction on $r$ gives $24+(r-5)=r+19$.
\end{proof}

Corollary~\ref{cor: d15 KUBT} gives
\begin{equation}\label{eq: d15 KUBT}
 f_{r-1}-(r+19)\geq 0,
 \qquad 5\leq r\leq 15,
\end{equation}
on every $r$-face of $P$.

\subsubsection{Vertex bounds}

We also need lower bounds on the number of vertices of faces of $P$. 
We first derive them under the more general assumption that every 
$3$-face has at least $q$ facets, where $q\geq 5$. Euler's relation 
and $2f_1\geq 3f_2$ imply that a $3$-polytope with at least $q$ facets has at least
\begin{equation}\label{eq: d15 nq}
 n_q:=\left\lceil\frac{q+4}{2}\right\rceil
\end{equation}
vertices. For an integer $a$, set
\begin{equation}\label{eq: d15 Kq}
 s_q(a):=\min\left\{5,
 \left\lfloor 2+\frac{2(a-2)}q\right\rfloor\right\},
 \qquad
 K_q(a):=\left\lceil\frac{q-s_q(a)+1}{2}\right\rceil.
\end{equation}

\begin{lemma}\label{lm: d15 crossing facet}
Let $H$ be a facet of an $r$-face $Q$ of $P$, where $r\geq 4$, and suppose
that $H$ has $a$ vertices. Then $Q$ has at least $a+K_q(a)$ vertices.
\end{lemma}

\begin{proof}
Choose a $3$-face $F$ of $H$, and write $b=f_0(F)$ and $m=f_2(F)\geq q$.
The average number of vertices of a $2$-face of $F$ is
\[
 \frac{2f_1(F)}m=2+\frac{2(b-2)}m
 \leq 2+\frac{2(a-2)}q.
\]
Every $3$-polytope also has a $2$-face with at most five vertices, by
applying the planar minimum-degree bound to its polar. Hence $F$ contains a
$2$-face $G$ with $s=f_0(G)\leq s_q(a)$.

The face figure $Q/G$ has the facet $H/G$, and therefore has a vertex
outside that facet. This vertex corresponds to a $3$-face $E$ of $Q$ such
that $E\cap H=G$. Let $t$ be the number of vertices of $E$ outside $H$. The
polygon $G$ is one facet of $E$ and has $s$ edges, while every other facet
of $E$ has at least three edges. Counting edge--facet incidences and using
Euler's relation gives
\[
 2f_1(E)\geq s+3(f_2(E)-1),
 \qquad f_2(E)\leq s+2t-1.
\]
Since $f_2(E)\geq q$, it follows that
\[
 t\geq \left\lceil\frac{q-s+1}{2}\right\rceil\geq K_q(a).
\]
The vertices of $H$ and the $t$ vertices outside it are distinct, proving
the claim.
\end{proof}

There is a second bound obtained by considering several $3$-faces through a
common $2$-face. Put $t=r-3$ and define
\begin{equation}\label{eq: d15 Sr}
 S_r(q):=\min\left\{
 n_q+t(n_q-3),\ q+2+t(n_q-4),\
 \left\lceil\frac{3q}{2}+2\right\rceil+t(n_q-5)
 \right\}.
\end{equation}

\begin{lemma}\label{lm: d15 common face}
Every $r$-face $Q$ of $P$, where $r\geq 4$, has at least $S_r(q)$ vertices.
\end{lemma}

\begin{proof}
Choose a $3$-face $F$ of $Q$, write $b=f_0(F)$, and choose a $2$-face $G$
of $F$ with $s=f_0(G)\leq s_q(b)$. The face figure $Q/G$ has dimension
$r-3$, and therefore has at least $r-2$ vertices. They correspond to
$r-2$ distinct $3$-faces through $G$, any two of which meet precisely in
$G$. Besides $F$, at least $t=r-3$ such faces remain. Each has at least
$n_q$ vertices, so their union has at least
\[
 b+t(n_q-s_q(b))
\]
vertices. The function $s_q(b)$ takes the values $3,4,5$. The smallest
possible values of $b$ in the three corresponding ranges are respectively
$n_q$, $q+2$, and $\lceil3q/2+2\rceil$. Minimizing over these three ranges
gives \eqref{eq: d15 Sr}.
\end{proof}

Define recursively
\begin{equation}\label{eq: d15 Mr}
 M_3(q):=n_q,
 \qquad
 M_r(q):=\max\left\{
 M_{r-1}(q)+K_q(M_{r-1}(q)),\ S_r(q)
 \right\}.
\end{equation}

\begin{proposition}\label{prop: d15 vertices}
If every $3$-face of $P$ has at least $q$ facets, then every $r$-face of
$P$, where $3\leq r\leq 15$, has at least $M_r(q)$ vertices. For $q=14$,
\begin{equation}\label{eq: d15 M values}
 \bigl(M_3(14),\ldots,M_{15}(14)\bigr)
 =(9,15,21,27,33,39,45,51,56,61,66,71,76).
\end{equation}
\end{proposition}

\begin{proof}
The case $r=3$ is \eqref{eq: d15 nq}. For $r\geq 4$, choose a facet $H$
of the given $r$-face. By induction, $f_0(H)\geq M_{r-1}(q)$. Since
$s_q(a+1)-s_q(a)\in\{0,1\}$, the function $a\mapsto a+K_q(a)$ is
nondecreasing. Lemma~\ref{lm: d15 crossing facet} therefore gives the first
term in the maximum in \eqref{eq: d15 Mr}, while
Lemma~\ref{lm: d15 common face} gives the second. This proves the
recurrence by induction. Substituting $q=14$ gives \eqref{eq: d15 M values}.
\end{proof}

The final conditional inequalities are the vertex bounds
\begin{equation}\label{eq: d15 V}
 f_0-M_r(14)\geq 0,
 \qquad 3\leq r\leq 15,
\end{equation}
on every $r$-face of $P$.

\subsection{The exact certificate and the proof of Theorem \ref{thm: d15 cap13}}
\label{subsec: exact certificate}

For an $r$-polytope, the $cd$-index satisfies
\[
 \phi_{dc^{r-2}}=f_0-2,
 \qquad
 \phi_{c^{r-2}d}=f_{r-1}-2.
\]
Consequently, the four types of conditional factors used in the
certificate have the following coordinates in the $cd$-quotient.

\begin{center}
\begin{tabular}{@{}lll@{}}
\toprule
bound & flag form & $cd$-quotient form \\
\midrule
facet bound, $r=3$ & $f_2-14$ &
$\phi_{cd}-12\phi_{c^3}$ \\
facet bound, $r=4$ & $f_3-15$ &
$\phi_{c^2d}-13\phi_{c^4}$ \\
facet bound, $5\leq r\leq 15$ & $f_{r-1}-(r+19)$ &
$\phi_{c^{r-2}d}-(r+17)\phi_{c^r}$ \\
vertex bound, $3\leq r\leq 15$ & $f_0-M_r(14)$ &
$\phi_{dc^{r-2}}-(M_r(14)-2)\phi_{c^r}$ \\
\bottomrule
\end{tabular}
\end{center}

The certificate records the four rows of this table under the internal
tags \texttt{CAP}, \texttt{K4}, \texttt{KUBT}, and \texttt{V}, respectively.
These tags serve only to identify the corresponding conditional factors.

The flag functionals allowed in the certificate are of two types. 
They are either unconditional rows or conditional rows of the form $H*G$, where $H$ is one of the conditional factors in the table and 
$G$ is an iterated convolution $G_1*\cdots*G_m$ of primitive unconditional factors of the complementary dimension. The convolution $G$ may also 
be empty. By \eqref{eq: d15 left anchor}, every row of either type is nonnegative on a polytope satisfying \eqref{eq: d15 counterexample}. Consequently, so is every positive combination of such rows, and (as explained in Section~\ref{subsec: why an lp}) it suffices to explicitly exhibit a positive combination that equals the constant functional $-\mathbf{1}$. 

\begin{proposition}[Exact rational certificate]
\label{prop: d15 certificate}
There are $987$ flag functionals $L_1,\ldots,L_{987}$ of these two types
and strictly positive rational numbers $\lambda_1,\ldots,\lambda_{987}$
such that the following identity of flag functionals holds:
\begin{equation}\label{eq: d15 certificate}
  \sum_{i=1}^{987}\lambda_iL_i=-\mathbf 1.
\end{equation}
\end{proposition}

\begin{remark}[Discovery of the certificate]
\label{rem: discovery of the certificate}
The identity \eqref{eq: d15 certificate} was found by searching, via
linear programming, for a nonnegative combination of admissible
inequalities equal to $-\mathbf 1$, in the sense made precise above in Subsection \ref{subsec: why an lp}.  In brief: the search began in dimension
$14$ at two independent thresholds ($q=27$ and $q=24$), each of which
produced a complete exact certificate; these were used to seed a descent
through intermediate thresholds down to $q=19$; a 
dimension-$14$ certificate was then lifted to dimension $15$ and the
threshold lowered again, from $19$ down to $16$; and finally a search over
the complete family of admissible dimension-$15$ inequalities produced
certificates at $q=15$ and, after a single row replacement eliminated four negative exact multipliers, at the target threshold $q=14$. Floating-point linear
programming (HiGHS) was used only to select candidate rows at each stage;
every reported multiplier was recomputed over $\mathbb Q$ and every
identity was verified exactly. The complete step-by-step account, including the precise linear program solved at each stage, is given in
Appendix~\ref {subsec: discovery lp formulation}.
\end{remark}

\begin{proof}[Computer-assisted proof of Proposition~\ref{prop: d15 certificate}] The machine-readable certificate, independent checker, and verification
report are available in the companion GitHub repository \url{https://github.com/shengtaoguo/d15-cap13-certificate}.
The companion repository contains the explicit $987$ inequalities $L_i$, an ordered
factor list for each $L_i$, and the exact rational multipliers $\lambda_i$ necessary for obtaining a certificate of type \ref{eq: why lp identity}.
The independent checker reconstructs every $L_i$ from its factor list using
\eqref{eq: d15 flag convolution}, verifies that each $L_i$ is either
unconditional or has the form $H*G$ described above, and checks that every
$\lambda_i$ is strictly positive. A verifier then checks
\eqref{eq: d15 certificate} holds in all $987$ degree-$15$ $cd$-coordinates.
Finally, in the original flag coordinates, the checker verifies that
$\mathbf 1+\sum_i\lambda_iL_i$ is a rational linear combination of the
generalized Dehn--Sommerville relations and therefore vanishes on every
polytope. All calculations in this verification are done in exact arithmetic using Python~3.10.
\end{proof}

% The machine-readable certificate, independent checker, and verification
% report are available in the companion GitHub repository \url{https://github.com/shengtaoguo/d15-cap13-certificate}.

% Run \texttt{python3 verification/verify.py} from the root of the repository.
% Python $3.10$ or latter version is required, but no external package is needed. A
% successful verification ends with a line beginning with \texttt{VERIFIED}.

Now, we can finish the proof of Theorem~\ref{thm: d15 cap13}.

\begin{proof}[Proof of Theorem~\ref{thm: d15 cap13}]
As observed at the beginning of the section, it suffices to work in
dimension $15$. Suppose that $P$ satisfies \eqref{eq: d15 counterexample}.
Every row in Proposition~\ref{prop: d15 certificate} is then nonnegative on
$P$. Evaluating \eqref{eq: d15 certificate} at $P$ gives
\[
 0\leq\sum_{i=1}^{987}\lambda_iL_i(P)=-\mathbf 1(P)=-1,
\]
a contradiction.
\end{proof}

\begin{remark}
The present certificate does not prove the next bound $12$. 
\end{remark}

\section{Acknowledgements}
\noindent\textbf{The role of AI.}
The Odin Automatic AI Research Agent was used as a sounding
board in the early stages of this project and suggested the initial proof behind the results which were later improved by the authors. The final proofs were written and verified by the authors. The Odin Automatic AI Research Agent was also used for generating the code for the computer-assisted portion of the proof in the paper.

\bibliographystyle{plain}
\bibliography{Biblio}

\newpage

\appendix

\section{Appendix: Details on the search for a certificate}
\label{subsec: discovery lp formulation}

The search for $L_i,\lambda_i$'s used in certificate \ref{eq: why lp identity} has the same linear algebra format in every
dimension, and its purpose is exactly the contradiction described
in subsection~\ref{subsec: why an lp}: writing $\mathcal W_d$ for the set of
degree-$d$ $cd$-words, each flag functional $L$ used at a given stage is
represented by its coefficient vector
$\boldsymbol{a}_{cd}(L)=(a_{L,w})_{w\in\mathcal W_d}\in
\mathbb Q^{\lvert\mathcal W_d\rvert}$; these vectors form the rows of a
matrix $A$; and, writing $\boldsymbol{e}_{c^d}$ for the vector with a $1$
in the $c^d$-coordinate and zeros elsewhere, the search looks for a
nonnegative multiplier vector $\boldsymbol\lambda$ with
\begin{equation}\label{eq: discovery dual}
 A^{\mathsf T}\boldsymbol{\lambda}=-\boldsymbol{e}_{c^d}.
\end{equation}
What changes between stages is only the dimension, the assumed threshold
$q$ on the number of facets of every $3$-face, and the available pool of
inequalities.

In dimension $14$, $\mathcal W_{14}$ has $610$ elements and the target
vector in \eqref{eq: discovery dual} is $-\boldsymbol{e}_{c^{14}}$. In
dimension $15$, $\mathcal W_{15}$ has $987$ elements and the target vector
is $-\boldsymbol{e}_{c^{15}}$. The complete dimension-$14$ systems at
$q=24$ and $q=27$ were assembled in the original flag coordinates together
with the generalized Dehn--Sommerville equalities; the searches at $q=23$
and $q=19$, and all the dimension-$15$ searches, used the $cd$-coordinates
directly.

The rows of $A$ were generated from the basic inequalities of
Sections~\ref{subsec: flag functionals and convolution}
and~\ref{subsec: counterexample inequalities}. An inequality on
$r$-polytopes was assigned grade $r+1$; convolution adds grades, so a
convolution produces an inequality on $d$-polytopes precisely when the
grades of its factors sum to $d+1$. The required total grade is therefore
$15$ in dimension $14$ and $16$ in dimension $15$. For example, in the
dimension-$15$ search the $3$-face inequality $f_2-q\geq0$ has grade $4$
and may be convolved with any ordered list of factors that are valid
without the temporary lower-bound assumption and whose grades sum to $12$.

All stages followed the same discovery-and-verification pattern. For the
current dimension and threshold, we formed a finite collection of valid
flag inequalities. The numerical searches were carried out with the HiGHS
linear-programming solver in floating-point arithmetic, using either its
interior-point method followed by its crossover procedure (which converts
the interior-point solution into a simplex basic solution), or its
dual-simplex method, depending on the stage. These computations were used
only to identify a small set of inequalities that might have nonzero
coefficients in a solution of \eqref{eq: discovery dual}; the numerical
coefficients were then discarded. For the selected set, the multipliers
were recomputed over $\mathbb Q$, and the resulting identity was verified
exactly. 

The subsections below specify, in each case, whether the starting
list came from an earlier certificate or from standard initial
inequalities.

\subsection{The initial dimension-$14$ certificates}
\label{subsec: discovery d14 search}

The first complete dimension-$14$ calculations were performed in the
original flag coordinates. The generalized Dehn--Sommerville relations and
the universal identity $f_{\varnothing}=1$ were imposed as equality
constraints, and a single variable $s$ relaxed every flag inequality by
the same amount. The same linear program can be written more economically
by parametrizing the solutions of the generalized Dehn--Sommerville
relations in $cd$-coordinates; the identity $f_{\varnothing}=1$ then
becomes $x_{c^d}=1$. In these coordinates, the linear program has the
equivalent reduced form
\begin{equation}\label{eq: discovery phase one}
 \min s
 \quad\text{subject to}\quad
 x_{c^d}=1,
 \quad A\boldsymbol{x}+s\mathbf 1\geq0,
 \quad s\geq0.
\end{equation}
A positive minimum shows that the unrelaxed inequalities are inconsistent
with $x_{c^d}=1$. The dimension-$14$ calculation first solved this program
at $q=14$ and $q=27$. Its minimum was zero at $q=14$, so that calculation
did not give a contradiction, whereas its minimum was positive at $q=27$.
To narrow the gap between these two results, the program next launched
complete calculations at the four intermediate values $q=16,19,21,24$ in
parallel. The calculation at $q=24$ was the next one to finish with a
positive minimum and an exact rational identity. Thus $q=27$ and $q=24$
provided two independently obtained starting certificates.

Here is how an exact identity was obtained at each of these two points.
Let $z$ denote the vector of the $2^{14}=16{,}384$ original flag
coordinates. In these coordinates, the complete system had the auxiliary
variable $s$, $1{,}144{,}309$ flag inequalities, and $61{,}441$ linear
equalities, including the generalized Dehn--Sommerville relations and the
universal identity $z_{\varnothing}=1$.

The original-coordinate system was solved through SciPy's
\texttt{linprog} interface to HiGHS~1.2.0. HiGHS first used its
interior-point algorithm to find a numerical minimum and then used its
simplex algorithm to obtain a numerical multiplier for each flag
inequality. The nonzero multipliers identified the inequalities retained
for the subsequent exact solution of \eqref{eq: discovery dual}. At both
$q=27$ and $q=24$, exactly $610$ of these multipliers were nonzero.

For each selected set of inequalities, the program solved the multiplier
equations again by Gaussian elimination with exact rational arithmetic,
taking advantage of the many zero entries in the equations. The
multipliers of the flag inequalities were required to be nonnegative,
while the multipliers of the generalized Dehn--Sommerville equalities were
unrestricted. Finally, the substitution of the rational multipliers gave
zero residual in all $2^{14}$ original flag coordinates, for both $q=27$
and $q=24$. When the generalized Dehn--Sommerville equalities are imposed,
the terms involving their multipliers vanish, and each result is an exact
identity involving $610$ flag inequalities. The two sets of $610$
inequalities were different. Their union contained $972$ distinct
inequalities, which, together with one standard inequality, formed the starting list for $q=23$.

\subsection{Lowering the dimension-$14$ threshold}
\label{subsec: discovery d14 descent}

Rather than solving the complete $q=23$ system from scratch, we began with
the $973$ inequalities and
added omitted inequalities only when they were violated by the current
feasible solution.

The $973$ inequalities did not yet constitute a certificate at $q=23$.
Each of the $972$ inherited inequalities came with an ordered list of the basic factors whose
convolution produced it. The program substituted $q=23$ into every factor
that depended on the threshold, repeated these convolutions, and retained
the standard inequality unchanged, thereby
obtaining $973$ inequalities valid under the new assumption. It first
solved \eqref{eq: discovery phase one} using only these inequalities. The
minimum was numerically zero, so there was a normalized candidate vector $\boldsymbol{x}$, with
$x_{c^d}=1$, satisfying all $973$ inequalities within numerical tolerance. 
Thus this initial list did not yet give a contradiction.

Recall that a possible counterexample must satisfy $L(\boldsymbol{x})\geq0$
for every available inequality $L$; the aim is to show that no normalized
vector $\boldsymbol{x}$ satisfies all these conditions. In the numerical
implementation, let $\widehat s$ denote the computed minimum in
\eqref{eq: discovery phase one}. The program treated a value
$\widehat s\leq\tau$, where $\tau=10^{-7}$, as a signal to continue the
search rather than as a proof of feasibility. The resulting numerical
vector was used only to find further valid inequalities that it violated.
At $q=23$, the program evaluated all $1{,}144{,}309$ available
inequalities at this vector. An inequality $L\geq0$ with
$L(\boldsymbol{x})<0$ was violated, and smaller values of
$L(\boldsymbol{x})$ represented larger violations. The same
constraint-generation scheme was used at $q=19$ and, with an implicit
search through the larger family, in the dimension-$15$ calculations
through $q=16$. Figure~\ref{fig:discovery-constraint-loop} shows the
common scheme.

Thus four successive batches enlarged the working list from $973$ to
$17{,}357$ inequalities. The next solve had positive minimum, and its
numerical multipliers selected $610$ inequalities. Recomputing those
multipliers over $\mathbb Q$ and substituting them in the original flag
coordinates verified an exact certificate at $q=23$. The $q=23$ certificate, together with the earlier $q=24$ certificate, was
then used to obtain a certificate at $q=19$.

A certificate at $q=19$ was already available, but we later carried out a separate
computation at $q=19$ as a check on the search, starting from $813$ seed
inequalities generated directly from the basic factors. The restricted
linear program was then solved nine times. Each of the first eight solves
had minimum zero, so a new batch of violated inequalities was identified and added.
The ninth solve, with a working list
of $27{,}847$ inequalities, had positive minimum. Its numerical
multipliers selected $610$ inequalities; recomputing their multipliers
over $\mathbb Q$ gave the exact identity at $q=19$. These $610$
inequalities were then passed to dimension $15$.

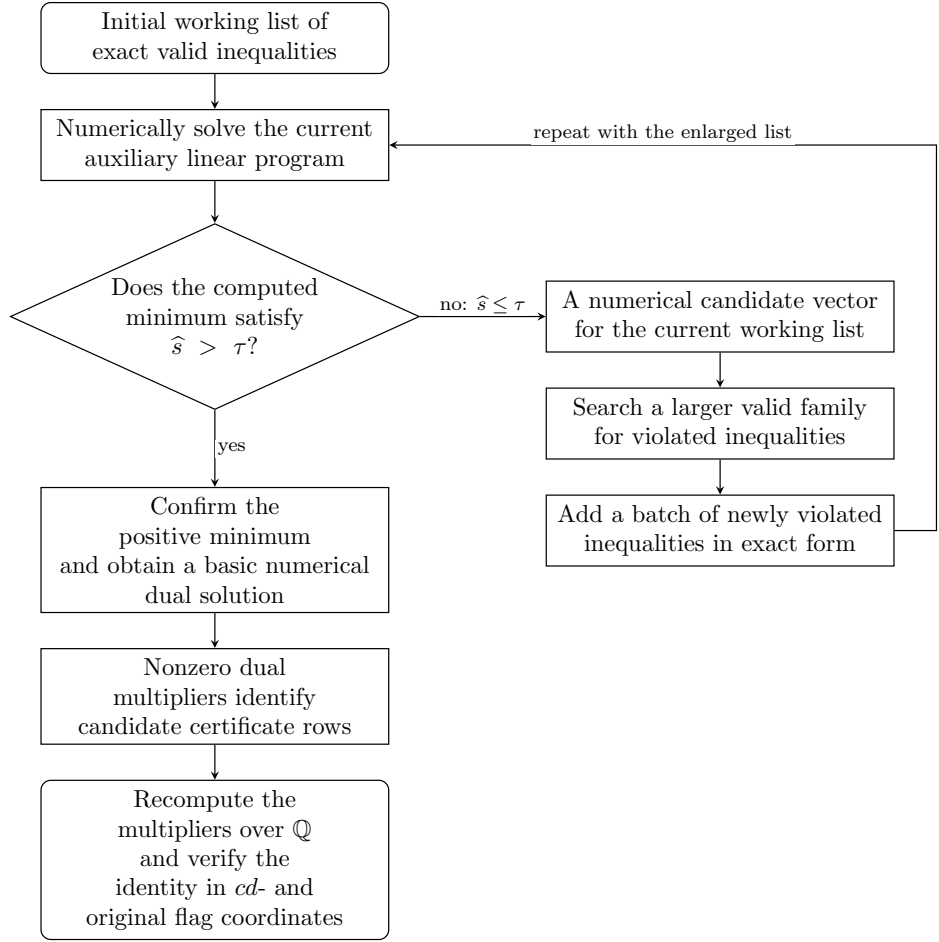
\begin{figure}[!t]
\centering
\resizebox{0.75\textwidth}{!}{%
\begin{tikzpicture}[
  >=stealth,
  node distance=5mm and 12mm,
  terminal/.style={draw, rounded corners=4pt, align=center, font=\small,
                   text width=47mm, inner sep=4pt,
                   execute at begin node={\hyphenpenalty=10000\relax}},
  process/.style={draw, rectangle, align=center, font=\small,
                  text width=47mm, inner sep=4pt,
                  execute at begin node={\hyphenpenalty=10000\relax}},
  decision/.style={draw, diamond, aspect=2.2, align=center, font=\small,
                   text width=31mm, inner sep=1.5pt,
                   execute at begin node={\hyphenpenalty=10000\relax}},
  lab/.style={font=\scriptsize, fill=white, inner sep=1pt},
  every path/.style={->, line width=0.5pt}
]
\node[terminal] (list) {Initial working list of\\exact valid inequalities};
\node[process, below=of list] (solve)
  {Numerically solve the current\\auxiliary linear program};
\node[decision, below=6mm of solve] (minimum)
  {Does the computed\\minimum satisfy\\$\widehat s>\tau$?};
\node[process, right=18mm of minimum] (feasible)
  {A numerical candidate vector\\for the current working list};
\node[process, below=of feasible] (scan)
  {Search a larger valid family\\for violated inequalities};
\node[process, below=of scan] (add)
  {Add a batch of newly violated\\inequalities in exact form};
\node[process, below=11mm of minimum] (confirm)
  {Confirm the positive minimum\\and obtain a basic numerical\\dual solution};
\node[process, below=of confirm] (select)
  {Nonzero dual multipliers identify\\candidate certificate rows};
\node[terminal, below=of select] (exact)
  {Recompute the multipliers over $\mathbb Q$\\and verify the identity in $cd$- and\\original flag coordinates};

\draw (list) -- (solve);
\draw (solve) -- (minimum);
\draw (minimum) -- node[lab,above] {no: $\widehat s\leq\tau$} (feasible);
\draw (feasible) -- (scan);
\draw (scan) -- (add);
\coordinate (loopright) at ([xshift=6mm]add.east);
\coordinate (looptop) at (loopright |- solve.east);
\draw (add.east) -- (loopright) -- (looptop)
  -- node[lab,midway,above] {repeat with the enlarged list} (solve.east);
\draw (minimum) -- node[lab,right] {yes} (confirm);
\draw (confirm) -- (select);
\draw (select) -- (exact);
\end{tikzpicture}
}
\caption{The common constraint-generation scheme used in the dimension-$14$
threshold descents at $q=23$ and $q=19$, and in the dimension-$15$ searches
from $q=19$ through $q=16$. Here $\widehat s$ is the computed minimum of
the auxiliary program and $\tau=10^{-7}$. Floating-point calculations
select rows only; the final identity is reconstructed and verified over
$\mathbb Q$.}
\label{fig:discovery-constraint-loop}
\end{figure}

\subsection{Lifting to dimension $15$ and lowering $q$ to $16$}
\label{subsec: discovery d15 search}

The passage to dimension $15$ had two steps: first the dimension-$14$
inequalities were lifted to dimension $15$, and then the threshold $q$ was
lowered successively from $19$ to $16$. We describe these two operations
separately.

Each of the $610$ inequalities in the dimension-$14$ certificate was
stored as an ordered list of factors whose convolution produced that
inequality. Such a list has total grade $15$. An inequality on a
$15$-polytope requires total grade $16$, so we inserted the trivial factor
$g_0=1$ on $0$-polytopes, which has grade $1$. The facet threshold
remained $q=19$ throughout this step. If a factor list has $m$ entries,
there are $m+1$ possible slots for $g_0$: before the first factor, between
two consecutive factors, or after the last factor. For each of the $610$
lists, we inserted $g_0$ into each possible slot and recomputed the
resulting convolution exactly. Because convolution is ordered, different
slots can produce different inequalities; conversely, different insertions
can also produce the same $cd$-coordinate vector. After equal vectors were
identified and duplicates removed, $1{,}999$ distinct dimension-$15$
inequalities remained. These were candidate inequalities for the new
search, not yet a dimension-$15$ certificate, and the dimension-$14$
multipliers were not carried over.

To initialize the dimension-$15$ search at $q=19$, we augmented the
$1{,}999$ lifted rows with $1{,}303$ grade-$16$ rows generated directly in
dimension $15$. Of these, $999$ were valid independently of the $3$-face
assumption. The remaining $304$ were obtained from $152$ unconditional
grade-$12$ functionals $G$ by forming both $(f_2-19)*G$ and $(f_0-5)*G$.
The latter family is redundant under $f_2\geq19$, but was included in the
search pool. The resulting restricted system therefore contained $3{,}302$
rows.

A floating-point solve with HiGHS produced nonzero dual multipliers on
$987$ rows. We retained these rows but discarded the numerical
multipliers. Solving the resulting $987$-by-$987$ system over $\mathbb Q$
and substituting the rational solution into \eqref{eq: discovery dual}
then verified the dimension-$15$ certificate at $q=19$ exactly.

The descents to $q=18,17,16$ applied the common loop in
Figure~\ref{fig:discovery-constraint-loop} one threshold at a time. For a
target value $q$, the exact certificate at $q+1$ was used only as a source
of ordered factor lists. The order of the factors was retained, while
every $q$-dependent factor and the resulting convolution were recomputed.
The preceding exact multipliers were not transported as coefficients of
the new certificate. Their values were used only to order the regenerated
seed rows, and, since no limit was placed on their number, every distinct
regenerated row was retained. After exact removal of duplicates and rows
already present among the $1{,}303$ standard seed rows, the preceding
certificates contributed $894$, $959$, and $961$ further rows at $q=18$,
$q=17$, and $q=16$, respectively.

Each restricted system was solved from scratch with the HiGHS
interior-point method. When $\widehat s\leq\tau$, the program evaluated
the available candidate inequalities at the resulting numerical vector.
Among these candidates, it retained only inequalities not already in the
working list and whose normalized violation exceeded the numerical
tolerance. If more than $16{,}384$ inequalities passed these two tests, it
added the $16{,}384$ most violated; otherwise, it added all of them. No
selected row was later discarded. The initial working lists at
$q=18,17,16$ therefore contained $2{,}197$, $2{,}262$, and $2{,}264$
inequalities, respectively. The loop reached $\widehat s>\tau$ on the
fifth, fifth, and seventh restricted solves; at those points the working
lists contained $65{,}448$, $67{,}798$, and $74{,}299$ inequalities,
respectively.

Once $\widehat s>\tau$, HiGHS solved the same restricted system again with
its crossover procedure enabled. This converts the interior-point solution
to a basic solution and supplies a sparse dual multiplier vector. At each
of the three thresholds, exactly $987$ dual multipliers were nonzero. The
floating-point values were then discarded. The corresponding square system
was solved over $\mathbb Q$, and the identity was checked exactly in all
$987$ $cd$-coordinates. An independent lift also verified the
corresponding identity in the original flag coordinates.

\subsection{The complete-list searches at $q=15$ and $q=14$}
\label{subsec: discovery d15 complete list}

At $q=15$, five rounds of row additions enlarged the working list to
$22{,}754$ inequalities. The sixth restricted solve again gave a computed
minimum essentially equal to zero. At the resulting numerical vector, the
separator found no new inequality in the then-current candidate family
whose normalized violation exceeded the numerical tolerance. This was not
an exact feasibility result; it showed only that the restricted
row-generation search had stalled.

We therefore replaced the preceding row-generation loop by a search over
every dimension-$15$ inequality obtainable from the basic factors
described above. We constructed this finite list in two parts. First, for
every $m\geq1$ and every ordered choice of primitive unconditional factors
$G_1,\ldots,G_m$ whose grades sum to $16$, we formed
$G_1*\cdots*G_m$. Second, for each conditional factor $H$ in
Section~\ref{subsec: counterexample inequalities} and every $m\geq0$, we
formed $H*G_1*\cdots*G_m$ for every ordered choice of primitive
unconditional factors such that the grades of $H,G_1,\ldots,G_m$ sum to
$16$. Thus $H$ may be followed by any number of unconditional factors,
including none. Its position first is required by
\eqref{eq: d15 left anchor}, because the counterexample assumption applies
to faces but not to quotients.

Every resulting coefficient vector was computed exactly in
$cd$-coordinates. When different factor lists produced the same vector, we
kept only one copy. This left $330{,}905$ unconditional vectors and
$36{,}844$ conditional vectors, for a total of $367{,}749$. The earlier
number $1{,}144{,}309$ refers instead to generated dimension-$14$ rows
represented in the original flag coordinates; the two numbers therefore
count different objects.

At $q=15$ and $q=14$ the search was allowed to use any of these
$367{,}749$ vectors.

At $q=15$ we began with the $987$ inequalities in the exact $q=16$
certificate, reconstructed from their recorded factor lists at $q=15$.
Solving the resulting $987$-by-$987$ instance of
\eqref{eq: discovery dual} over $\mathbb Q$ gave a unique multiplier
vector with $498$ negative components. Thus the preceding certificate
could not be reused unchanged at the new threshold. HiGHS then used its
dual-simplex algorithm to replace inequalities, drawing from the complete
list, until every component of the numerical multiplier vector was
nonnegative. To favor inequalities already selected, their objective
coefficients were set to zero, while those of all other inequalities were
set to one. Thus the calculation retained as many previously selected
inequalities as possible while allowing any inequality in the complete
list to replace one of them. The calculation at $q=15$ produced an exact
identity involving $987$ inequalities with positive coefficients.

At $q=14$ we likewise began with the $987$ inequalities in the exact
$q=15$ certificate, reconstructed at $q=14$. Their multipliers did not
preserve nonnegativity, so the calculation was repeated over the complete
list. When the first numerically selected set was solved exactly, four
rational coefficients were negative, although their decimal
approximations were close to zero. The program then tested every
replacement obtained by removing one of the $987$ selected inequalities
and adding one unselected inequality. One replacement made all $987$
rational coefficients strictly positive. These inequalities and
coefficients are those used in Proposition~\ref{prop: d15 certificate}.

\end{document}